\documentclass{amsart}
\usepackage{color, framed}
\usepackage{amssymb,amsmath,amsthm,amsfonts,framed}
\usepackage{graphicx,caption,subcaption}
\usepackage{appendix}

 \usepackage[colorlinks=true]{hyperref}
\hypersetup{urlcolor=blue, citecolor=red}
\usepackage{hyperref}
\usepackage{verbatim}
\def\R {\mathbb{R}}

\def\N {\mathbb{N}}
\def\S {\mathbb{S}}
\def\Z {\mathbb{Z}}

\def\eps{\varepsilon}

\def\pa{\partial}

\newcommand{\loc}{\mathrm{loc}}

\renewcommand{\div}{\mathrm{div}}

\newcommand{\de}[1] {\mathrm{d} #1}

\newcommand{\ddfrac}[2] {\frac{\displaystyle #1 }{\displaystyle #2} }

\DeclareMathOperator{\inter}{Int}
\DeclareMathOperator{\supp}{supp}

\newtheorem{proposition}{Proposition}[section]
\newtheorem{theorem}[proposition]{Theorem}
\newtheorem{corollary}[proposition]{Corollary}
\newtheorem{lemma}[proposition]{Lemma}

\theoremstyle{definition}

\newtheorem{remark}[proposition]{Remark}
\newtheorem*{plan}{Plan of the paper}
\newtheorem*{notation}{Notation}

\numberwithin{equation}{section}

\title[Exponential entire solutions]{Entire solutions with exponential growth for an elliptic system modeling phase-separation}

\begin{document}
\maketitle

{\footnotesize
\centerline{Nicola Soave}
 \centerline{Universit\`a degli Studi di Milano Bicocca - Dipartimento di Ma\-t\-ema\-ti\-ca e Applicazioni}
   \centerline{Via Roberto Cozzi 53, 20125 Milano, Italy}
%\{Dipartimento di Matematica e Applicazioni, Universit\`a degli Studi
%di Milano-Bicocca, Via Bicocca degli Arcimboldi, 8, 20126 Milano,
%Italy}
\centerline{email: n.soave@campus.unimib.it}

\medskip

\centerline{Alessandro Zilio}
 \centerline{Politecnico di Milano - Dipartimento di Ma\-t\-ema\-ti\-ca ``Francesco Brioschi"}
   \centerline{Piazza Leonardo da Vinci 32, 20133 Milano, Italy}
%\{Dipartimento di Matematica e Applicazioni, Universit\`a degli Studi
%di Milano-Bicocca, Via Bicocca degli Arcimboldi, 8, 20126 Milano,
%Italy}
\centerline{email: alessandro.zilio@mail.polimi.it}
}

 \begin{abstract}
\noindent We prove the existence of entire solutions with exponential growth for the semilinear elliptic system
\[
\begin{cases}
-\Delta u = -u v^2 & \text{in $\R^N$}\\
-\Delta v= -u^2 v & \text{in $\R^N$} \\
u,v>0,
\end{cases}
\]
for every $N \ge 2$. Our construction is based on an approximation procedure, whose convergence is ensured by suitable Almgren-type monotonicity formulae. The construction of \emph{some} solutions is extended to systems with $k$ components, for every $k > 2$.
\end{abstract}


\section{Introduction and main results}
In this paper we investigate the existence of entire solutions with exponential growth for the semilinear elliptic system
\begin{equation}\label{eqn: system R}
    \begin{cases}
    - \Delta u = - u v^2 \\
    - \Delta v = - u^2 v \\
    u,v > 0,
    \end{cases}
\end{equation}
in $\R^2$ (thus in $\R^N$ for every $N \ge 2$). System \eqref{eqn: system R}, which appears in the study of phase-separation phenomena for Bose-Einstein condensates with multiple states, has been intensively studied in the last years; we refer in particular to \cite{BeLiWeZh, CaffLin, ChanLinLinLin, DaWaZh, NoTaTeVe, TaTe}, where physical motivations are discussed and a precise description of the phase-separation is derived, and to \cite{BeLiWeZh, BeTeWaWe} where existence and qualitative properties of entire solutions are central topics. In \cite{NoTaTeVe}, it is proved that if $(u,v)$ is an entire solution to \eqref{eqn: system R} and is globally $\alpha$-H\"older continuous for some $\alpha \in (0,1)$, then one between $u$ and $v$ is constant while the other is identically $0$. On the other hand, in \cite{BeLiWeZh} the authors show that there exists a nontrivial solution for the system of ODEs
\[
\begin{cases}
    - u'' = - u v^2 & \text{in $\R$}\\
    - v'' = - u^2 v & \text{in $\R$}\\
    u,v>0
    \end{cases}
\]
which is reflectionally symmetric with respect to a point of $\R$, in the sense that there exists $t_0 \in \R$ such that $u(t_0+t)=v(t_0-t)$ for every $t \in \R$, and has linear growth: there exists $C>0$ such that
\[
u(t)+v(t) \le C(1+|t|) \qquad \forall t \in \R.
\]
The paper \cite{BeTeWaWe} completes the study of the $1$-dimensional problem with the proof of the uniqueness of the positive $1$-dimensional profile, up to translations and scalings. Always in \cite{BeTeWaWe}, the authors construct entire solutions to \eqref{eqn: system R} with algebraic growth for any integer rate of growth greater then $1$; here and in the rest of the paper we say that $(u,v)$ has algebraic growth if there exist $p \ge 1$ and $C>0$ such that
\[
u(x)+v(x) \le C(1+|x|^p) \qquad \forall x \in \R^N.
\]
The solutions constructed in \cite{BeTeWaWe} are not $1$-dimensional, and \emph{are modeled on}  (we will be more precise later, see Remark \ref{rmk:model}) the homogeneous harmonic polynomials $\Re(z^d)$, for every $d \ge 2$. There is a deep relationship between entire solutions to \eqref{eqn: system R} and harmonic functions; this relationship has been established in \cite{DaWaZh,NoTaTeVe}. For instance, in case $(u,v)$ has algebraic growth, it is possible to show that up to a subsequence, the blow-down family, defined by
\[
\left( u_R(x), v_R(x) \right) = \frac{R^{N-1}}{\int_{\pa B_R(0)} u^2 + v^2} \left( u(Rx), v(Rx)\right),
\]
is uniformly convergent in every compact subset of $\R^N$, as $R \to +\infty$, to a limiting profile $(\Psi^+,\Psi^-)$, where $\Psi$ is a homogeneous harmonic polynomial (see Theorem 1.4 in \cite{BeTeWaWe}).

To conclude this bibliographic introduction, we have to mention that major efforts have been done recently in order to prove classification results and in particular the $1$-dimensional symmetry of solutions to \eqref{eqn: system R}. This is motivated by the relationship between \eqref{eqn: system R} and the Allen-Cahn equation, which has been established in \cite{BeLiWeZh}, and led the authors to formulate a De Giorgi's-type and a Gibbons'-type conjecture for solutions to \eqref{eqn: system R}; for results in this direction, we refer to \cite{BeLiWeZh,BeTeWaWe, Fa, FaSo, Wa}.

\medskip

Motivated by the quoted achievements, we wonder if the system \eqref{eqn: system R} has solutions with super-algebraic growth. We can give a positive answer to this question proving the existence of solutions with exponential growth. In our construction we adapt the same line of reasoning introduced in the proof of Theorem 1.3 of \cite{BeTeWaWe}. Therein, the authors proved the existence of solutions to \eqref{eqn: system R} with the same symmetry of the function $\Re(z^d)$ in any bounded ball $B_R(0) \subset \R^2$, with boundary conditions $u=(\Re(z^d))^+$, $v=(\Re(z^d))^-$ on $\pa B_R(0)$. By means of suitable monotonicity formulae, they could pass to the limit for $R \to +\infty$ obtaining convergence (up to a subsequence) for the previous family to a nontrivial entire solution. In this sense, their solutions \emph{are modeled on} the harmonic functions $\Re(z^d)$.

Here, having in mind the construction of solutions with exponential growth, and recalling the relationship between entire solution of our system and harmonic functions, we start by considering
\[
\Phi(x,y):=\cosh x \sin y.
\]
The first of our main results is the following.

\begin{theorem}\label{main theorem 1}
There exists an entire solution $(u,v) \in (\mathcal{C}^\infty(\R^2))^2$ to system \eqref{eqn: system R} such that
\begin{itemize}
\item[1)] $u(x,y+2\pi)=u(x,y)$ and $v(x,y+2\pi)=v(x,y)$,
\item[2)] $u(-x,y)=u(x,y)$ and $v(-x,y)=v(x,y)$,
\item[3)] the symmetries
\begin{align*}
v(x,y)=u(x,y-\pi) \quad & \quad u(x,\pi-y)= v(x,\pi +y)\\
u\left(x,\frac{\pi}{2}+y\right)= u\left(x,\frac{\pi}{2}-y\right) \quad & \quad v\left(x,\frac{3}{2}\pi+y\right)= v\left(x,\frac{3}{2}\pi-y\right)
\end{align*}
hold,
\item[4)] $u-v>0$ in $\{\Phi>0\}$ and $v-u>0$ in $\{\Phi<0\}$,
\item[5)] $u > \Phi^+$ and $v > \Phi^-$ in $\R^2$,
\item[6)] the function (\emph{Almgren quotient})
\[
r \mapsto \frac{\int_{(0,r) \times (0,2\pi)} |\nabla u|^2 + |\nabla v|^2 + 2 u^2 v^2}{\int_{\{r\} \times [0,2\pi]} u^2+v^2}
\]
is well-defined for every $r>0$, is nondecreasing, and
\[
\lim_{r \to +\infty} \frac{\int_{(0,r) \times (0,2\pi)} |\nabla u|^2 + |\nabla v|^2 + 2 u^2 v^2}{\int_{\{r\} \times [0,2\pi]} u^2+v^2} = 1,
\]
\item[7)] %there exists $C>0$ such that
%\[
%\pi \cosh^2 r \le \int_{\{r\} \times [0,2\pi]} u^2+v^2\le C e^{2r} \qquad \forall r>1;
%\]
%moreover,
there exists the limit
\[
\lim_{r \to +\infty} \frac{\int_{\{r\} \times [0,2\pi]} u^2+v^2}{e^{2r}}=: \alpha \in (0,+\infty).
\]
\end{itemize}
\end{theorem}

\begin{remark}\label{rmk:model}
This solution \emph{is modeled on} the harmonic function $\Phi$, in the sense that it inherits the symmetries of $(\Phi^+,\Phi^-)$ and has the same rate of growth of $\Phi$.
\end{remark}

\begin{remark}\label{rmk: exponential growth}
Point 7) of the Theorem gives a lower and a upper bound to the rate of growth of the quadratic mean of $(u,v)$ on $\{r\} \times [0,2\pi]$ when $r$ varies:
\[
\left(\int_{\{r\} \times [0,2\pi]} u^2+v^2 \right)^{\frac{1}{2}}= O(e^r) \qquad \text{as $r \to +\infty$}.
\]
The domain of integration takes into account the periodicity of $(u,v)$. The quadratic mean of $(u,v)$ on $\{r\} \times [0,2\pi]$ has exponential growth, and the rate of growth is the same of the function $e^r$, which in turns has the same rate of growth of $\Phi$. Note that the coefficient $1$ in the exponent of $e^r$ coincides with the limit as $r \to +\infty$ of the Almgren quotient defined in point 6).
\end{remark}

\begin{remark}\label{rem:scaling invariance}
With a scaling argument, it is not difficult to prove the existence of entire solutions with exponential growth of order $\lambda$ for every $\lambda>0$ (in the previous sense). To see this, let
\[
\left(u_{\lambda}(x,y), v_\lambda(x,y)\right)= \left(\lambda u (\lambda x, \lambda y), \lambda v (\lambda x, \lambda y\right).
\]
It is straightforward to check that $(u_\lambda, v_\lambda)$ is still a solution to \eqref{eqn: system R} in the plane, is $\frac{2\pi}{\lambda}$-periodic in $y$ and is such that
\[
u_\lambda(x,y) \ge \lambda\left(\cosh (\lambda x)\sin(\lambda y)\right)^+ \quad \text{and} \quad v_\lambda(x,y) \ge \lambda \left(\cosh (\lambda x)\sin(\lambda y)\right)^-.
\]
Moreover,
\begin{equation}\label{eq10}
\lim_{r \to +\infty} \frac{\int_{(0,r) \times \left(0,\frac{2\pi}{\lambda}\right)} |\nabla u_\lambda|^2 + |\nabla v_{\lambda}|^2 + 2 u_{\lambda}^2 v_{\lambda}^2}{\int_{\{r\} \times \left[0,\frac{2\pi}{\lambda}\right]} u_{\lambda}^2+v_{\lambda}^2} = \lambda,
\end{equation}
and
\[
\lim_{r \to +\infty} \frac{ \int_{\{r\} \times \left[0,\frac{2\pi}{\lambda}\right]} u_\lambda^2+v_\lambda^2}{e^{2 \lambda r}}  = \lambda \alpha.
\]
One can consider the solution $(u_\lambda,v_\lambda)$ as related to the harmonic function $\cosh(\lambda x) \sin (\lambda y)$. This reveals that there exists a correspondence
\[
\{(u_\lambda,v_\lambda):\lambda >0\} \leftrightarrow \{ \sin(\lambda x) \cosh(\lambda y): \lambda>0\}.
\]
Due to the invariance under translations and rotations of problem \eqref{eqn: system R}, the family $\{(u_\lambda,v_\lambda):\lambda >0\}$ can equivalently be related with the families of harmonic functions $\{\cosh(\lambda x) \left[C_1 \cos(\lambda y) + C_2 \sin(\lambda y)\right]\}$ or $\{\left[ C_3\cos(\lambda x) + C_4 \sin(\lambda x)\right] \cosh(\lambda y): \lambda>0\}$, where $C_1,C_2,C_3,C_4 \in \R$.

As observed in Remark \ref{rmk: exponential growth}, the limit of the Almgren quotient in \eqref{eq10} describes the rate of the growth of the quadratic mean of $(u_\lambda,v_\lambda)$ computed on an interval of periodicity in the $y$ variable. The previous computation reveals that for every $\lambda>0$ we can construct a solution having rate of growth equal to $\lambda$. This marks a relevant difference between entire solutions with polynomial growth and entire solutions with exponential growth: while in the former case the admissible rates of growth are quantized (Theorem 1.4 of \cite{BeTeWaWe}), in the latter one we can prescribe any positive real value as rate of growth.
\end{remark}

Remark \ref{rem:scaling invariance} reveals that, starting from the solution found in Theorem \ref{main theorem 1}, we can build infinitely-many entire solutions with different exponential growth. However, noting that system \ref{eqn: system R} is invariant under rotations, translations and scalings, intuitively speaking they are all the same solution. We wonder if there exists an entire solution of \eqref{eqn: system R} having exponential growth which cannot be obtained by that found in Theorem \ref{main theorem 1} through a rotation, a translation or a scaling; the answer is affirmative. We denote
\[
\Gamma(x,y):= e^x \sin y.
\]

\begin{theorem}\label{thm: main thm 2}
There exists an entire solution $(u,v) \in (\mathcal{C}^\infty(\R^2))^2$ to system \eqref{eqn: system R}  which enjoys points 1), 3), 4) of Theorem \ref{main theorem 1}; moreover
\begin{itemize}
\item[2)] for every $r \in \R$
\begin{equation}\label{eqn:finite_energy}
\int_{(-\infty,r) \times (0,2\pi) } |\nabla u|^2 + |\nabla v|^2 +  u^2 v^2 <+\infty,
\end{equation}
\item[5)] $u> \Gamma^+$ and $v> \Gamma^-$ in $\R^2$,
\item[6)] the function (\emph{Almgren quotient})
\[
r \mapsto \frac{\int_{(-\infty,r) \times (0,2\pi)} |\nabla u|^2 + |\nabla v|^2 + 2 u^2 v^2}{\int_{\{r\} \times (0,2\pi)} u^2+v^2}
\]
is well-defined for every $r>0$, is nondecreasing, and
\[
\lim_{r \to +\infty} \frac{\int_{(-\infty,r) \times (0,2\pi)} |\nabla u|^2 + |\nabla v|^2 + 2 u^2 v^2}{\int_{\{r\} \times (0,2\pi)} u^2+v^2} = 1,
\]
\item[7)] %there exists $C>0$ such that
%\[
%\pi  e^{2r} \le \int_{\{r\} \times [0,2\pi]} u^2+v^2\le C e^{2r} \qquad \forall r>1;
%\]
%moreover,
there exist the limits
\[
\lim_{r \to +\infty} \frac{\int_{\{r\} \times [0,2\pi]} u^2 +v^2}{e^{2r}}=:\beta \in (0,+\infty) \quad \text{and} \quad \lim_{r \to -\infty} \int_{\{r\} \times [0,2\pi]} u^2 +v^2=0.
\]
\end{itemize}
\end{theorem}

\begin{remark}
This solution is modeled on the harmonic function $\Gamma$. As explained in Remark \ref{rmk: exponential growth}, it is possible to obtain a family of entire solutions which is in correspondence with a family of harmonic functions.
\end{remark}

\begin{remark}
Note that the Almgren quotients that we defined in Theorem \ref{main theorem 1} and \ref{thm: main thm 2} are different. They are both different to the Almgren quotient which has been defined in \cite{BeTeWaWe}. %Our idea is that it is possible to associate to each entire solution of \eqref{eqn: system R} an Almgren quotient (whose definition takes into account the periodicity and the symmetries of the solution itself) which provides information about the rate of growth of the solution. We point out that at this stage this is only a conjecture.
\end{remark}

We can partially generalize our existence result to the case of systems with many components. To be precise, given an integer $k$, we will construct a solution $(u_1, \dots, u_k)$ of
\begin{equation}\label{eqn: system k comp}
	\begin{cases}
	- \Delta u_i = - u_i \sum_{j\neq i} u_j^2 \\
	u_i >0 ,
	\end{cases} i=1, \ldots, k,
\end{equation}
in the whole plane $\R^2$ having the same growth and the same symmetries of $\Gamma$. Here and in the paper we consider the indexes $\mod k$.

\begin{theorem}\label{main theorem 2}
There exists an entire solution $(u_1,\dots,u_k) \in (\mathcal{C}^\infty(\R^2))^k$ to system \eqref{eqn: system k comp} such that, for every $i=1,\ldots,k$,
\begin{itemize}
\item[1)] $u_i(x,y+ k \pi)=u_i(x,y)$,
\item[2)] the symmetries
\[
u_{i+1}(x,y) = u_i\left(x,y- \pi\right) \quad \quad u_{1}\left(x,\frac{\pi}{2}+y\right)= u_1\left(x,\frac{\pi}{2}-y\right)
\]
hold,
\item[3)] for every $r \in \R$
\[
\int_{(-\infty,r) \times (0,k\pi)} \sum_{i=1}^k |\nabla u_i|^2 + \sum_{1 \le i < j \le k} u_i^2 u_j^2 <+\infty;
\]
\item[4)] the function (\emph{Almgren quotient})
\[
r \mapsto \frac{\int_{(-\infty,r) \times (0,k\pi)} \sum_{i=1}^k |\nabla u_i|^2 +  2 \sum_{1 \le i < j \le k} u_i^2 u_j^2}{ \int_{\{r\} \times [0,k\pi]} \sum_{i=1}^k u_i^2   }
\]
is well-defined for every $r>0$, is nondecreasing, and
\[
\lim_{r \to +\infty} \frac{\int_{(-\infty,r) \times (0,k\pi)} \sum_{i=1}^k |\nabla u_i|^2 +  2 \sum_{1 \le i < j \le k} u_i^2 u_j^2}{ \int_{\{r\} \times [0,k\pi]} \sum_{i=1}^k u_i^2   } = 1.
\]
\item[5)] %for every $\eps>0$ there exist $r_0>0$ such that
%\[
%C_1 e^{2(1-\eps) r}\le \int_{\{r\} \times [0,k\pi]} \sum_{i=1}^k u_i^2  \le C_2 e^{ 2r} \qquad \forall r >r_0,
%\]
%where $C_1,C_2>0$. Moreover,
there exist the limits
\[
\lim_{r \to +\infty} \int_{\{r\} \times [0,k\pi]} \sum_{i=1}^k u_i^2 =: \gamma \in (0,+\infty) \quad \text{and} \quad      \lim_{r \to -\infty} \int_{\{r\} \times [0,k\pi]} \sum_{i=1}^k u_i^2 =0.
\]
\end{itemize}
\end{theorem}

This solution is modeled on $\Gamma$.

\medskip

Our last main result is the counterpart of Theorem 1.4 of \cite{BeTeWaWe} in our setting. This can be quite surprising because, as we already observed, we cannot expect a quantization of the admissible rates of growth dealing with solutions with exponential growth, see Remark \ref{rem:scaling invariance}. Nevertheless, if we consider solutions which are periodic in one component, prescribing a period such a quantization can be recovered.

%This can be quite unexpected because, as we already observed, dealing with entire solutions with exponential growth we do not have a quantization of the rates of growth of the solutions. However, if we assume the solution of \eqref{eqn: system R} to be $2\pi$-periodic in $y$ (here $2\pi$ is not necessarily the minimal period), we can recover the same quantization, as explained in the following statements.

\begin{theorem}\label{thm: blow-down 2}
Let $(u,v)$ be a nontrivial solution of \eqref{eqn: system R} in $\R^2$ which is $2\pi$-periodic in $y$, and such that one of the following situation occurs:
\begin{itemize}
\item[($i$)] there holds
\[
\lim_{r \to -\infty} \int_{\{r\} \times [0,2\pi] } u^2 + v^2 = 0,
\]
and
\[
    d:=\lim_{r \to +\infty} \frac{\int_{(-\infty,r) \times (0,2\pi)} |\nabla u|^2 +|\nabla v|^2 + u^2 v^2   }{\int_{\{r\} \times [0,2\pi]} u^2 +v^2    } < +\infty.
\]
\item[($ii$)] $\pa_x u = 0 = \pa_x v$ on $\{a\} \times [0,2\pi]$ for some $a \in \R$, and
\[
    d:=\lim_{r \to +\infty} \frac{\int_{(a,r) \times (0,2\pi)} |\nabla u|^2 +|\nabla v|^2 + u^2 v^2   }{\int_{\{r\} \times [0,2\pi]} u^2 +v^2    } < +\infty.
\]
\end{itemize}
Then $d$ is a positive integer,
\[
   \left(\int_{\{r\} \times [0,2\pi] } u^2 + v^2 \right)^{\frac{1}{2}} = O(e^{dr}) \qquad \text{as $r \to +\infty$},
\]
and the sequence
\[
(u_R(x,y),v_R(x,y)):= \frac{1}{ \sqrt{\int_{\{r\} \times [0,2\pi] } u^2 + v^2 }} \left(u(x+R,y), v(x+R,y) \right)
\]
converges in $\mathcal{C}^0_{\loc}(\R^2)$ and in $H^1_{\loc}(\R^2)$ to $(\Psi^+,\Psi^-)$, where $\Psi(x,y) = e^{dx} \left( C_1 \cos(d y)+ C_2 \sin(d y)\right)$ for some $C_1,C_2 \in \R$.
\end{theorem}

\begin{notation}
We will deal with functions defined in domains of type $(a,b) \times \R$, where $a<b$ are extended real numbers ($a=-\infty$ and $b=+\infty$ are admissible). We will often assume that $(u_1,\dots,u_k)$ is $k\pi$-periodic in $y$; therefore, we can think to $(u_1,\dots, u_k)$ as defined on the cylinder
\[
C_{(a,b)} := (a,b) \times \S_k \quad \text{where} \quad \S_k=\R/(k\pi \Z).
\]
We will also denote $\Sigma_r :=\{r\} \times \S_k$. In case $b>0$, $a=-b$, we will simply write $C_b$ instead of $C_{(-b,b)}$ to simplify the notation.
\end{notation}

\begin{plan}
In section \ref{sec:monotonicity} we will prove some monotonicity formulae which will come useful in the rest of the paper. We can deal with two types of solutions: solutions satisfying a homogeneous Neumann condition defined in a cylinder $C_{(a,b)}$ with $a>-\infty$, or solutions defined in a semi-infinite cylinder of type $C_{(-\infty,b)}$ and decaying at $x \to -\infty$. For the sake of completeness and having in mind to use some monotonicity formulae in the proof of Theorem \ref{main theorem 2}, we will always consider the case of systems with $k$ components.

The proof of Theorem \ref{main theorem 1} will be the object of section \ref{sec:proof1}. It follows the same sketch of the proof of Theorem 1.3 in \cite{BeTeWaWe}: we start by showing that for any $R>0$ there exists a solution $(u_R,v_R)$ to \eqref{eqn: system R} in the cylinder $C_R$, with Dirichlet boundary condition
\[
u_R= \Phi^+ \quad \text{and} \quad v_R= \Phi^- \quad \text{on $\{-R,R\} \times [0,2\pi]$},
\]
and exhibiting the same symmetries of $(\Phi^+,\Phi^-)$. In order to obtain a solution defined in the whole $C_\infty$, we wish to prove the $\mathcal{C}_{loc}^2(C_\infty)$ convergence of the family $\{(u_R,v_R): R>1\}$, as $R \to +\infty$. To show that this convergence occurs, we will exploit the monotonicity formulae proved in subsection \ref{sub:monot Neumann}. With respect to Theorem 1.3 of \cite{BeTeWaWe}, major difficulties arise in the precise characterization of the growth of $(u,v)$, points 6) and 7) of Theorem \ref{main theorem 1}.

In section \ref{sec:sec2} we will prove Theorem \ref{thm: main thm 2}. One could be tempted to try to adapt the proof of Theorem \ref{main theorem 1} replacing $\Phi$ with $\Gamma$. Unfortunately, in such a situation we could not exploit the results of subsection \ref{sub:monot Neumann}; this is related to the lack of the even symmetry in the $x$ variable of the function $\Gamma$ (note that the function $\Phi$ enjoys this symmetry). A possible way to overcome this problem is to work in semi-infinite cylinders $C_{(-\infty,R)}$ and use the monotonicty formulae proved in subsection \ref{sub:monot finite}. But to work in an unbounded set introduces further complications: for instance, the compactness of the Sobolev embedding and of some trace operators, a property that we will use many times in section \ref{sec:proof1}, does not hold in $C_{(-\infty,R)}$. Although we believe that this kind of obstacle can be overcome, we propose a different approach for the construction of solutions modeled on $\Gamma$, which is based on the elementary limit
\[
\lim_{R \to +\infty}  \Phi_R(x,y)  =\Gamma(x,y) \qquad \forall (x,y) \in \R^2,
\]
where $\Phi_R(x,y) = 2 e^{-R} \cosh(x+R) \sin y$. We will prove the existence of a solution $(u_R,v_R)$ of \eqref{eqn: system R} in $C_{(-3R,R)}$ with Dirichlet boundary condition
\[
u_R= \Phi_R^+ \quad \text{and} \quad v_R= \Phi_R^- \quad \text{on $\{-3R,R\} \times [0,2\pi]$},
\]
and exhibiting the same symmetries of $(\Phi_R^+,\Phi_R^-)$. Then, using again the results of section \ref{sec:monotonicity}, we will pass to the limit as $R \to +\infty$ proving the compactness of $\{(u_R,v_R)\}$.

Section \ref{sec:k comp} is devoted to the study of systems with many components. As in \cite{BeTeWaWe} the authors could prove in one shot an existence theorem for $2$ or $k$ components (there are no substantial changes in the proofs), it is natural to wonder if here we can simply adapt step by step the construction carried on in section \ref{sec:proof1} or \ref{sec:sec2}, or not. Unfortunately, the answer is negative: following the sketch of the proof of Theorem \ref{main theorem 1}, we can adapt most the results of sections \ref{sec:proof1} and \ref{sec:sec2} with minor changes, but in the counterpart of Proposition \ref{existence thm in bdd cylinders} we cannot prove the pointwise estimate given by point 4). As a consequence, with respect to subsections \ref{sub:compat 1} and \ref{sub:compat 2} we cannot show that the limit of the sequence $(u_{1,R},\ldots,u_{k,R})$ does not vanish. Note that, in the case of two components, this nondegeneracy is ensured precisely by the above pointwise estimate. As far as the case of $k$ component in \cite{BeTeWaWe}, we observe that they obtained nondegeneracy through their Corollary 5.4, which is the counterpart of point ($i$) of our Corollary \ref{doubling}. But, while therein the estimate of the growth given by this statement is optimal, in our situation it does not provide any information; this is related to the different expression of the term of rest in the Almgren monotonicity formula, Proposition \ref{prp:Almgren Neumann}. This is why we have to use a completely different argument which is not based on the existence of solutions for the system of $k$ components in bounded cylinders (or in semi-infinite cylinders), but rests on Theorem 1.6 of \cite{BeTeWaWe}. Roughly speaking, we will obtain the existence of a solution of \eqref{eqn: system k comp} with exponential growth as a limit of solutions of the same system having algebraic growth.

The proof of Theorem \ref{thm: blow-down 2} will be the object of section \ref{sec: blow-down}.

We conclude the paper with an appendix, in which we state and prove some known results for which we cannot find a proper reference.

\end{plan}

\section{Almgren-type monotonicity formulae}\label{sec:monotonicity}

Let $k \ge 2$ be a fixed integer. In this section we are going to prove some monotonicity formulae for solutions of
\begin{equation}\label{eqn:system k}
\begin{cases}
-\Delta u_i= - u_i \sum_{j \neq i} u_j^2 \\
u_i>0
\end{cases}
\end{equation}
defined in a cylinder $C_{(a,b)}$ (this means that we assume from the beginning that $(u_1,\dots,u_k)$ is $k\pi$-periodic in $y$).

In this section we will use many times the following general result:

\begin{lemma}\label{lem: Pohozaev harmonic k}
Let $(u_1, \dots, u_k)$ be a solution of \eqref{eqn: system k comp} in $C_{(a,b)}$. Then the function
\[
    r \mapsto  \int_{\Sigma_{r}} \sum_{i=1}^k |\nabla u_i|^2+ \sum_{1 \le i<j \le k} u_i^2u_j^2 - 2 \int_{\Sigma_{r}}  \sum_{i=1}^k (\partial_x u_i)^2
\]
is constant in $(a,b)$.
\end{lemma}
\begin{proof}
Let $a<r_1<r_2<b$. We test the equation \eqref{eqn:system k} with $(\pa_x u_1,\dots, \pa_x u_k)$ in $C_{(r_1,r_2)}$: for every $i$ it results
\[
\int_{C_{(r_1,r_2)}} \frac{1}{2} \pa_x \left(|\nabla u_i|^2\right) + \left(\sum_{j \neq i}  u_j^2\right) u_i \pa_x u_i = \int_{\Sigma_{r_2}} (\pa_x u_i)^2 - \int_{\Sigma_{r_1}} (\pa_x u_i)^2.
\]
Summing for $i=1,\dots,k$ we obtain
\[
\int_{C_{(r_1,r_2)}} \pa_x\left( \sum_{i} |\nabla u_i|^2+ \sum_{i<j} u_i^2 u_j^2\right) = 2\int_{\Sigma_{r_2}} \sum_ i (\pa_x u_i)^2 -2 \int_{\Sigma_{r_1}} \sum_i (\pa_x u_i)^2,
\]
which gives the thesis.
\end{proof}

\subsection{Solutions with Neumann boundary conditions}\label{sub:monot Neumann}

In this subsection we are interested in solutions to \eqref{eqn:system k} defined in $C_{(a,b)}$ (thus $k\pi$-periodic in $y$), with $a>-\infty$ and $b \in (a,+\infty]$, and satisfying a homogeneous Neumann boundary condition on $\Sigma_{a}$, that is,
\begin{equation}\label{eqn:Neumann assumption}
\pa_x u_i = 0 \qquad \text{on $\Sigma_{a}$, for every $i=1,\dots,k$}.
\end{equation}
Firstly, we observed that under this assumption Lemma \ref{lem: Pohozaev harmonic k} implies

\begin{lemma}\label{lem:pohozaev Neumann}
Let $(u_1,\ldots,u_k)$ be a solution of \eqref{eqn:system k} in $C_{(a,b)}$, such that \eqref{eqn:Neumann assumption} holds true. For every $r \in (a,b)$ the following identity holds:
\[
\int_{\Sigma_r} \sum_{i=1}^k |\nabla u_i|^2 + \sum_{1 \le i<j \le k} u_i^2 u_j^2=  2 \int_{\Sigma_r} \sum_{i=1}^k (\pa_x u_i)^2
       + \int_{\Sigma_{a}} \sum_{i=1}^k (\pa_y u_i)^2 + \sum_{1 \le i<j \le k} u_i^2 u_j^2.
\]\end{lemma}

For a solution $(u_1,\dots,u_k)$ of \eqref{eqn:system k} in $C_{(a,b)}$ satisfying \eqref{eqn:Neumann assumption}, we define
\[
\begin{split}
E^{sym}(r)& := \int_{C_{(a,r)}} \sum_{i=1}^k |\nabla u_i|^2 + 2\sum_{1 \le i<j \le k} u_i^2 u_j^2, \\
\mathcal{E}^{sym}(r) & := \int_{C_{(a,r)}} \sum_{i=1}^k |\nabla u_i|^2 + \sum_{1 \le i<j \le k} u_i^2 u_j^2, \\
H(r)  & :=\int_{\Sigma_r} \sum_{i = 1}^k u_i^2
\end{split}
\]

\begin{remark}
The index $sym$ denotes the fact that, as we will see, the quantities $E^{sym}$ and $ \mathcal{E}^{sym}$ are well suited to describe the growth of the solution $(u_1,\dots,u_k)$ only if $(u_1,\dots,u_k)$ satisfies the \eqref{eqn:Neumann assumption}, which can be considered as a symmetry condition. Indeed, under \eqref{eqn:Neumann assumption} one can extend $(u_1,\dots,u_k)$ on $C_{(2a-b,b)}$ by even symmetry in the $x$ variable.
\end{remark}

\medskip

By regularity, $E$, $\mathcal{E}$ and $H$ are smooth. A direct computation shows that they are nondecreasing functions: in particular
\begin{equation}\label{eqn:der H}
H'(r) = 2 \int_{\Sigma_r} \sum_i u_i \pa_\nu u_i = 2 E(r),
\end{equation}
where the last identity follows from the divergence theorem and the boundary conditions of $(u_1,\dots,u_k)$. Our next result consist in showing that also the ratio between $E$ (or $\mathcal{E}$) and $H$ is nondecreasing.

\begin{proposition}\label{prp:Almgren Neumann}
Let $(u_1,\ldots,u_k)$ be a solution of \eqref{eqn:system k} in $C_{(a,b)}$ such that \eqref{eqn:Neumann assumption} holds true. The \emph{Almgren quotient}
\[
N^{sym}(r):= \frac{E^{sym}(r)}{H(r)}
\]
is well defined and nondecreasing in $(a,b)$. Moreover
\[
\int_a^r \frac{\int_{\Sigma_s} \sum_{i<j} u_i^2 u_j^2}{H(s)}\, \de s \le N(r).
\]
Analogously, the function (which we will call \emph{Almgren quotient}, too) $\displaystyle
\mathfrak{N}^{sym}(r) := \frac{E^{sym}(r)}{H(r)}$ is well defined and nondecreasing in $(a,b)$, and
\[
\mathfrak{N}'(r) \ge 2 \mathfrak{N}(r) \frac{\int_{C_{(a,r)}} \sum_{i<j} u_i^2 u_j^2 }{H(r)} + 2 \left( \frac{\int_{C_{(a,r)}} \sum_{i<j} u_i^2 u_j^2 }{H(r)} \right)^2.
\]
\end{proposition}

In the rest of this subsection we will briefly write $E,\mathcal{E},N$ and $\mathfrak{N}$ instead of $E^{sym}, \mathcal{E}^{sym}, N^{sym}$ and $\mathfrak{N}^{sym}$ to ease the notation.
\begin{proof}
Since $(u,v) \in H^1_{\loc}(C_{(a,b)})$ is nontrivial, $E$ and $H$ are positive in $(a,b)$ and bounded for $r$ bounded. We compute, by means of Lemma \ref{lem:pohozaev Neumann}
\[
    \begin{split}
    E'(r) &=  \int_{\Sigma_r} \sum_i |\nabla u_i|^2 + 2 \sum_{i<j} u_i^2 u_j^2\\
          &=  \int_{\Sigma_r} 2\sum_i (\pa_x u_i)^2 + \sum_{i<j} u_i^2 u_j^2 + \int_{\Sigma_a} \sum_i (\pa_y u_i)^2 + \sum_{i<j} u_i^2 u_j^2.
    \end{split}
\]
Note that $\pa_x u_i = \pa_\nu u_i$ on $\Sigma_r$. Using the previous identity and the \eqref{eqn:der H} we are in position to compute the logarithmic derivative of $N$:
\begin{align*}
    \frac{N'(r)}{N(r)} &= \frac{E'(r)}{E(r)} - \frac{H'(r)}{H(r)} \\
    &= 2\frac{ \int_{\Sigma_r} \sum_i (\pa_\nu u_i)^2}{ \int_{\Sigma_r} \sum_i u  \pa_\nu u_i} + \frac{ 2\int_{\Sigma_a} \sum_i (\pa_y u_i)^2 + \sum_{i<j}  u_i^2u_j^2+ \int_{\Sigma_r} \sum_{i<j} u_i^2 u_j^2}{ E(r)} -2 \frac{ \int_{\Sigma_r} \sum_ i u \pa_\nu u_i}{\int_{\Sigma_r} \sum_i u_i^2}\\
    &\geq 2 \left(\frac{ \int_{\Sigma_r} \sum_i (\pa_\nu u_i)^2}{ \int_{\Sigma_r} \sum_i u  \pa_\nu u_i} - \frac{ \int_{\Sigma_r} \sum_ i u \pa_\nu u_i}{\int_{\Sigma_r} \sum_i u_i^2} \right) + \frac{\int_{\Sigma_r} \sum_{i<j} u_i^2 u_j^2}{ E(r)} \ge \frac{\int_{\Sigma_r} \sum_{i<j} u_i^2 u_j^2}{ E(r)} \ge 0,
\end{align*}
where we used the Cauchy-Schwarz and the Young inequalities.  As a consequence, $N$ is nondecreasing in $(a,b)$. Note also that
\[
N'(r) \ge \frac{\int_{\Sigma_r} \sum_{i<j}u_i^2 u_j^2}{ H(r)} \quad \Rightarrow\quad  N(r) \ge \int_a^r \frac{\int_{\Sigma_s} \sum_{i<j} u_i^2 u_j^2 }{H(s)}\,\de s
\]
for every $r>a$. The same argument can be adapted with minor changes to prove the monotonicity of $\mathfrak{N}$.
\end{proof}

As a first consequence, we have the following

\begin{corollary}\label{doubling}
Let $(u_1,\dots, u_k)$ be a solution of \eqref{eqn:system k} in $C_{(a,b)}$ such that \eqref{eqn:Neumann assumption} holds.
\begin{itemize}
\item[($i$)] If $N(r) \ge \underline{d}$ for $r\ge s >a$, then
\[
\frac{H(r_1)}{e^{2 \underline{d} r_1}} \le \frac{H(r_2)}{e^{2 \underline{d} r_2}}  \qquad \forall \ s \le r_1<r_2<b,
\]
\item[$ii$)] If $N(r) \le \overline{d}$ for $r \le t<b$, then
\[
 \frac{H(r_1)}{e^{2 \overline{d} r_1}} \ge \frac{H(r_2)}{e^{2 \overline{d}r_2}} \qquad \forall \ a < r_1<r_2 \le t.
\]
\end{itemize}
\end{corollary}
\begin{proof}
We prove only ($ii$). Recalling that $H'(r)= 2E(r)$ (see \eqref{eqn:der H}), we have
\[
\frac{\de}{\de r} \log H(r) = 2 N(r) \le 2\overline{d} \qquad \forall r \in (a,t].
\]
By integrating, the thesis follows.
\end{proof}

The next step is to prove a similar monotonicity property for the function $E$. Our result rests on Theorem 5.6 of \cite{BeTeWaWe} (see also \cite{BeLiWeZh}), which we state here for the reader's convenience

\begin{theorem}\label{thm:teo 5.6 k}
Let $k$ be a fixed integer and let $\Lambda>1$. Let
\[
\mathcal{L}(k,\Lambda): = \min \left\{ \int_{0}^{2\pi} \sum_{i=1}^k (f'_i)^2 +\Lambda \sum_{1 \le i<j \le k} f_i^2 f_j^2 \left| \begin{array}{l}
f_1,\dots, f_k \in H^1([0,2\pi]), \ \int_{0}^{2\pi} \sum_{i=1}^k f_i^2 = 1 \\
f_{i+1}(t)= f_i\left(t-\frac{2\pi}{k}\right), \ f_1(\pi +t )= f_1(\pi-t)
\end{array}\right.
\right\},
\]
where the indexes are counted $\mod k$. There exists $C>0$ such that
\[
\left( \frac{k}{2} \right)^2-C \Lambda^{- 1/4} \le \mathcal{L}(k,\Lambda) \le \left(\frac{k}{2} \right)^2.
\]
\end{theorem}

\begin{remark}\label{rem: su teo 5.6 k}
Having in mind to apply Theorem \ref{thm:teo 5.6 k} on $2\pi$-periodic functions, note that the condition $f_1(\pi + t)=f_1(\pi-t)$ can be replaced by $f_1(t+\tau)=f_1(\tau -t)$ for any $\tau \in [0,2\pi)$.
\end{remark}

For a fixed $r_0 \in (a,b)$, let us introduce
\[
	\varphi(r;r_0):= \int_{r_0}^r \frac{\de s}{H(s)^{1/4}}.
\]
The function $\varphi$ is positive and increasing in $\R^+$; thanks to point ($i$) of Corollary \ref{doubling} and to the monotonicity of $N$, whenever $(u,v)$ is nontrivial $\varphi$ is bounded by a quantity depending only $H(r_0)$ and $N(r_0)$. To be precise:
\begin{equation}\label{eqn:estimate for phi}
	\varphi(r;r_0) \leq 2\frac{e^{\frac{1}{2}N(r_0)r_0}}{H(r_0)^{\frac14} N(r_0)} \left[e^{-\frac12 N(r_0) r_0}-e^{-\frac12 N(r_0)r} \right].
\end{equation}
This, together with the monotonicity of $\varphi(\cdot;r_0)$, implies that if $b =+\infty$ then there exists the limit
\begin{equation}\label{eqn: limit of fi}
\lim_{r \to +\infty} \varphi(r;r_0)<+\infty.
\end{equation}

\begin{lemma}\label{lem monotonicity for E}
Let $(u_1,\dots,u_k)$ be a solution of \eqref{eqn: system R} in $C_{(a,b)}$ such that \eqref{eqn:Neumann assumption} holds. Let $r_0 \in (a,b)$, and assume that
\begin{equation}\label{symmetry u-v}
u_{i+1}(x,y)=u_i(x,y-\pi) \quad \text{and} \quad u_1\left(x,\tau+y \right)= u_1\left(x,\tau -y\right)
\end{equation}
where $\tau \in [0,k\pi)$. There exists $C>0$ such that the function $\displaystyle{r \mapsto \frac{E(r)}{e^{2r}} e^{ C \varphi(r;r_0)}}$ is nondecreasing in $r$ for $r>r_0$.
\end{lemma}

\begin{proof}
Recalling the \eqref{eqn:der H}, we compute the logarithmic derivative
\begin{equation}\label{eqn: der E k}
\frac{\de}{\de r} \log\left(\frac{E(r)}{e^{2r}} \right) = -2 + \frac{\int_{\Sigma_r} \sum_i \left( \pa_\nu u_i\right)^2 + \int_{\Sigma_r} \left(\pa_y u_i\right)^2 + 2 \sum_{i<j} u_i^2 u_j^2   }{\int_{\Sigma_r} \sum_i u_i \pa_\nu u_i  }
\end{equation}
To apply Theorem \ref{thm:teo 5.6 k}, we observe that $\Sigma_r=\{r\} \times [0,k\pi]$, so that
\begin{multline}\label{eqn:der E k 2}
\int_{\Sigma_r} \left(\pa_y u_i\right)^2 + 2 \sum_{i<j} u_i^2 u_j^2 = \int_0^{k\pi} \left(\pa_y u_i(r,y)\right)^2 + 2 \sum_{i<j} u_i(r,y)^2 u_j(r,y)^2 \, \de y \\
= \frac{2}{k} \int_0^{2\pi} \left(\pa_y \tilde u_i(r,y)\right)^2 + 2 \left(\frac{k}{2}\right)^2 \sum_{i<j} \tilde u_i(r,y)^2 \tilde u_j(r,y)^2 \, \de y,
\end{multline}
where $\tilde u_i(r,y) = u_i \left(r,\frac{k}{2}y\right)$. By a scaling argument, thanks to assumption \eqref{symmetry u-v} (see also Remark \ref{rem: su teo 5.6 k}) we can say that for every $\Lambda> \frac12$ there holds
\begin{multline*}
\int_0^{2\pi} \left(\pa_y \tilde u_i(r,y)\right)^2 + \left(\frac{k}{2}\right)^2\frac{2\Lambda}{\int_0^{2\pi} \sum_i \tilde u_i(r,y)^2\,\de y} \sum_{i<j} \tilde u_i(r,y)^2 \tilde u_j(r,y)^2\, \de y  \\
\ge \mathcal{L}\left(k,2\Lambda \left(\frac{k}{2}\right)^2\right) \int_0^{2\pi} \sum_i \tilde u_i(r,y)^2\, \de y = \frac{2}{k} \mathcal{L}\left(k,2\Lambda \left(\frac{k}{2}\right)^2\right) \int_{\Sigma_r} \sum_i  u_i^2
 \end{multline*}
The choice
 \[
 \Lambda = \int_0^{2\pi} \sum_i \tilde u_i(r,y)^2\,\de y = \frac{2}{k}H(r)
 \]
yields
\[
\int_0^{2\pi} \left(\pa_y \tilde u_i(r,y)\right)^2 + 2 \left(\frac{k}{2}\right)^2 \sum_{i<j} \tilde u_i(r,y)^2 \tilde u_j(r,y)^2 \, \de y \ge  \frac{2}{k} \mathcal{L}\left(k,k H(r) \right) \int_{\Sigma_r} \sum_i  u_i^2,
\]
and coming back to \eqref{eqn:der E k 2} we obtain
\[
\int_{\Sigma_r} \left(\pa_y u_i\right)^2 + 2 \sum_{i<j} u_i^2 u_j^2 \ge \left(\frac{2}{k} \right)^2 \mathcal{L}\left(k,k H(r) \right) \int_{\Sigma_r} \sum_i  u_i^2.
\]
Plugging this estimate into the \eqref{eqn: der E k} we see that
\[
\begin{split}
\frac{\de}{\de r} \log\left(\frac{E(r)}{e^{2r}} \right) & \ge -2 + \frac{ \int_{\Sigma_r} \sum_i \left( \pa_\nu u_i\right)^2  + \left(\frac{2}{k} \right)^2 \mathcal{L}\left(k,k H(r) \right) \int_{\Sigma_r} \sum_i  u_i^2}{\int_{\Sigma_r} \sum_i u_i \pa_\nu u_i  } \\
& \ge -2 + 2\frac{2}{k} \sqrt{ \mathcal{L}\left(k,k H(r) \right) } \ge -\frac{C}{H(r)^{1/4}}
\end{split}
\]
where we used Theorem \ref{thm:teo 5.6 k}. An integration gives the thesis.
\end{proof}

\begin{lemma}\label{cor estimate for N}
Let $(u_1,\dots,u_k)$ be a nontrivial solution of \eqref{eqn:system k} in $C_{(a,+\infty)}$,  and assume that \eqref{eqn:Neumann assumption} and \eqref{symmetry u-v} hold. If $d:=\lim_{r \to +\infty} N(r)<+\infty$, then $d \ge 1$ and
\[
\lim_{r \to +\infty}\frac{E(r)}{e^{2r}}>0.
\]
\end{lemma}
\begin{proof}
Let us fix $r_0>a$. Firstly, from the previous Lemma and the \eqref{eqn: limit of fi}, we deduce that there exists the limit
\[
l:= \lim_{r \to+\infty} \frac{E(r)}{e^{2r}} \ge 0.
\]
Recalling that $\varphi(r;r_0)$ is bounded, it results
\[
\frac{E(r)}{e^{2r}} \ge e^{-C \varphi(r;r_0)} \frac{E(r_0)}{e^{2 r_0}} \ge C>0 \qquad \forall r > r_0,
\]
so that the value $l$ is strictly greater then $0$. Now, assume by contradiction that $d= \lim_{r \to +\infty} N(r)<1$. The monotonicity of $N$ implies $N(r) \le d$ for every $r>0$. Hence, from Corollary \ref{doubling} we deduce
\[
\frac{H(r)}{e^{2dr}} \le \frac{H(r_0)}{e^{2d r_0}} \quad \forall r>r_0 \quad \Rightarrow \quad \limsup_{r \to +\infty} \frac{H(r)}{e^{2dr}} < +\infty \quad \Rightarrow \quad \lim_{r \to +\infty} \frac{H(r)}{e^{2r}}=0,
\]
which in turns gives
\[
0<l=\lim_{r \to +\infty} \frac{E(r)}{e^{2r}}=\lim_{r \to +\infty} N(r) \lim_{r \to +\infty} \frac{H(r)}{e^{2r}} =0,
\]
a contradiction.
\end{proof}

\subsection{Solutions with finite energy in unbounded cylinders}\label{sub:monot finite}

In what follows we consider a solution $(u_1,\dots,u_k)$ of \eqref{eqn:system k} defined in an unbounded cylinder $C_{(-\infty,b)}$, with $b \in \R$ (the choice $b=+\infty$ is admissible). In this setting we assume that $(u_1,\dots,u_k)$ has a sufficiently fast decay as $x \to -\infty$, in the sense that
\begin{equation}\label{eqn:decay}
H(r):= \int_{\Sigma_r} \sum_{i=1}^k u_i^2 \to 0 \quad \text{as $r \to -\infty$}.
\end{equation}

First of all, we can show that under assumption \eqref{eqn:decay} $(u_1,\dots,u_k)$ has finite energy in $C_{(-\infty,b)}$.

\begin{lemma}\label{lem:decay to finite E}
Let $(u_1,\dots,u_k)$ be a solution of \eqref{eqn: system k comp} in $C_{(-\infty,b)}$, such that \eqref{eqn:decay} holds. Then
\[
\mathcal{E}^{unb}(r):=\int_{C_{(-\infty,r)}} \sum_{i=1}^k |\nabla u_i|^2 + \sum_{1 \le i<j \le k} u_i^2 u_j^2 < +\infty \qquad \forall r <b.
\]
\end{lemma}

The index $unb$ stands for the fact that the energy is evaluated in an unbounded cylinder, and will be omitted in the rest of the subsection.

\begin{proof}
Firstly, being a solution in $C_{(-\infty,b)}$, it results $(u_1,\dots, u_k) \in H^1_{loc}(C_{(-\infty,b)})$. Thus, under assumption \eqref{eqn:decay}, there exists $C>0$ such that $H(r) \le C$ for every $r<b$.

Let $r_0<b$. Let us introduce, for $r >0$, the functional
\[
    e(r) := \int_{C_{(-r+r_0,r_0)}} \sum_i |\nabla u_i|^2+ \sum_{i<j} u_i^2u_j^2.
\]
For the sake of simplicity, in the rest of the proof we assume $r_0=0$ (thus $b>0$). By direct computation and an application of Lemma \ref{lem: Pohozaev harmonic k}, we find
\[
        e'(r) = \int_{\Sigma_{-r}} \sum_i |\nabla u_i|^2+ \sum_{i<j} u_i^2u_j^2 = 2 \int_{\Sigma_{-r}} (\partial_x u_i)^2 +\int_{\Sigma_{0}} \sum_i |\nabla u_i|^2+ \sum_{i<j} u_i^2u_j^2 - 2 \int_{\Sigma_{0}} (\partial_x u_i)^2
\]
that is
\begin{equation}\label{eqn: identita nella pohozaev k-comp}
    \int_{\Sigma_{-r}} (\partial_x u_i)^2  = \frac12 e'(r) + C_0
\end{equation}
On the other hand, testing the equation \eqref{eqn: system k comp} in $C_{(-r,0)}$ by $(u_1,\dots,u_k)$ and summing for $i = 1,\dots, k$, we find
\[
    \begin{split}
        e(r) &\leq \int_{C_{(-r,0)}} \sum_i |\nabla u_i|^2 + 2 \sum_{i<j} u_i^2u_j^2 = \int_{\Sigma_0}\sum_i u_i \partial_{x} u_i - \int_{\Sigma_{-r}} \sum_i u_i \partial_{x} u_i  \\
        &\leq \int_{\Sigma_0}\sum_i u_i \partial_{x} u_i + \left( \int_{\Sigma_{-r} } \sum_i (\partial_x u_i)^2 \right)^\frac12 \left(\int_{\Sigma_{-r} } \sum_i u_i^2 \right)^\frac12
    \end{split}
\]
Let us assume that by contradiction that $e(r) \to +\infty$ as $r \to +\infty$. Taking the square of the previous inequality, using the boundedness of $H$ and the assumption \eqref{eqn:decay}, we have
\[
    \begin{cases}
        \frac{1}{C^2} (e(r) + C_1)^2 -2C_0 \leq e'(r) &\text{ for $r > \bar r$}\\
        e(\bar r) > 0,
    \end{cases}
\]
for some $C_0,C_1>0$ and $\bar r$ sufficiently large. Any solution to the previous differential inequality blows up in finite time, in contradiction with the fact that $(u_1,\dots,u_k) \in H^1_{\loc}(C_{(-\infty,b)})$. As a consequence $e$ is bounded and, by regularity,
\[
    \int_{C_{(-\infty,r)}} \sum_i |\nabla u_i|^2+ \sum_{i<j} u_i^2u_j^2 < +\infty \qquad \forall r<b. \qedhere
\]
\end{proof}

\begin{remark}\label{decay energy}
As a byproduct of the previous Lemma, if $(u_1,\dots, u_k)$ solves the \eqref{eqn: system k comp} in $C_{(-\infty,b)}$ and \eqref{eqn:decay} holds, then
\[
\lim_{r \to -\infty} \mathcal{E}(r) = 0.
\]
\end{remark}

Having in mind to recover the monotonicity formulae of the previous subsection in the present situation, we cannot adapt the proof of Lemma \ref{lem:pohozaev Neumann}, where assumption \eqref{eqn:Neumann assumption} played an important role. However, we can obtain a similar result with a different proof.

\begin{lemma}\label{lem:pohozaev energy}
Let $(u_1,\ldots,u_k)$ be a solution to \eqref{eqn: system R} in $C_{(-\infty,b)}$, such that \eqref{eqn:decay} holds. Then
\[
\int_{\Sigma_r} \sum_{i=k} |\nabla u_i|^2 + \sum_{1 \le i<j le k} u_i^2 u_j^2 = 2\int_{\Sigma_r} \sum_{i=1}^k (\pa_x u_i)^2
\]
for every $r <b$.
\end{lemma}
\begin{proof}
We use the method of the variations of the domains: for $\psi \in \mathcal{C}^1_c(-\infty,r)$, we consider
\[
u_{i,\eps}(r,y)=u_i(r+\eps \psi(r),y) \qquad i=1,\ldots,k.
\]
It is possible to see $(u_{1,\eps},\ldots,u_{k,\eps})$ as a smooth variations of $(u_1,\dots,u_k)$ with compact support in $C_{(-\infty,r)}$: indeed
\[
u_i(x+\eps \psi(x),y)-u_i(x,y)= \eps \pa_x u(\xi_x,y) \psi(x),
\]
where $\xi_x \in (x,x+\eps \psi(x))$. To proceed, we explicitly remark that any solution to \eqref{eqn: system k comp} is critical for the energy functional
\[
J(v_1,\dots,v_k):= \int_{C_{(-\infty,b)}} \sum_{i=1}^k |\nabla v_i|^2 + \sum_{1 \le i<j \le j} v_i^2 v_j^2
\]
with respect to variations with compact support in $\mathcal{C}^\infty_c(C_{(-\infty,b)})$. Note that $J(u_1,\dots,u_k) = \mathcal{E}(b)$. As $(u_1,\dots,u_k)$ is a smooth solution of \eqref{eqn: system k comp} with finite energy $\mathcal{E}(r)$, it follows that
\begin{equation}\label{eqn:var1}
\begin{split}
0 & = \lim_{\eps \to 0} \frac{ \int_{C_{(-\infty,r)}} \sum_i |\nabla u_{i,\eps}|^2 + \sum_{i<j} u_{i,\eps}^2 u_{j,\eps}^2- \mathcal{E}(r)}{\eps} \\
& = \int_{C_{(-\infty,r)}}\frac{\pa}{\pa \eps}  \left. \left( \sum_i |\nabla u_i (x+\eps \psi(x),y)|^2 + \sum_{i<j} u_i^2(x+\eps \psi (x),y) u_j^2 (x+\eps \psi(x),y)\right)\right|_{\eps=0} \, \de x \de y \\
& \quad + 2\lim_{\eps \to 0} \int_{C_{(-\infty,r)}} \psi'(x) \sum_i (\pa_x u_i)^2(x+ \eps \psi (x)) \, \de x \de y \\
&= \int_{C_{(-\infty,x)}} \left(2 \sum_i (\pa_x u_{i})^2 - \left(\sum_i |\nabla u_i|^2+ \sum_{i<j} u_i^2 u_j^2\right)\right) \psi'
\end{split}
\end{equation}
for every $\psi \in \mathcal{C}_c^1(-\infty,x)$. Since $\mathcal{E}(r)<+\infty$, for every $\eps>0$ there exists a compact $K_\eps \subset C_{(-\infty,r)}$ such that
\[
\int_{C_{(-\infty,r)} \setminus K_\eps} \sum_i |\nabla u_i|^2  + \sum_{i<j} u_i^2 u_j^2 <\eps.
\]
Let $\psi \in \mathcal{C}^1(-\infty,r)$ be such that $\| \psi\|_{\mathcal{C}^1(-\infty,r)} < +\infty$ and $\psi=0$ in a neighborhood of $r$.  It is possible to write $\psi=\psi_1+\psi_2$ where $\psi_1 \in \mathcal{C}_c^1(-\infty,r)$ and $\supp \psi_2 \times (\R/k\pi\Z) \subset (C_{(-\infty,r)} \setminus K_\eps)$. Therefore, from \eqref{eqn:var1} it follows
\begin{multline*}
\int_{C_{(-\infty,r)}}    \left(2\sum_i  (\pa_x u_i)^2 - \left(\sum_i |\nabla u_i|^2 + \sum_{i<j} u_i^2 u_j^2\right)\right) \psi' \\
= \int_{C_{(-\infty,r)} \setminus K_\eps} \left(2 \sum_i (\pa_x u_i)^2  - \left(\sum_i |\nabla u|^2 + \sum_{i<j} u_i^2 u_j^2\right)\right) \psi_2' \\
 \le 3\| \psi\|_{\mathcal{C}^1(-\infty,x)} \int_{C_{(-\infty,r)} \setminus K_\eps} \left(\sum_i |\nabla u_i|^2 + \sum_{i<j} u_i^2 u_j^2\right) < C \eps.
\end{multline*}
Since $\eps$ has been arbitrarily chosen, we obtain
\begin{equation}\label{eqn:var2}
\int_{C_{(-\infty,r)}} \left(2\sum_i  (\pa_x u_i)^2 - \left(\sum_i |\nabla u_i|^2 + \sum_{i<j} u_i^2 u_j^2\right)\right) \psi'=0
\end{equation}
for every $\psi \in \mathcal{C}^1(-\infty,r)$ be such that $\| \psi\|_{\mathcal{C}^1(-\infty,r)} < +\infty$ and $\psi=0$ in a neighborhood of $r$.\\
Now, let $\psi \in \mathcal{C}^1((-\infty,r])$ be such that $\|\psi\|_{\mathcal{C}^1((-\infty,r])}<+\infty$. For a given $\eps>0$, we introduce a cut-off function $\eta \in \mathcal{C}^\infty(\R)$ such that
\[
\eta(s)= \begin{cases} 1 & \text{if $s \le r-\eps$} \\
0 & \text{if $s \ge r$}.
\end{cases}
\]
Since $\eta \psi \in \mathcal{C}^1(-\infty,r)$, $\|\eta \psi\|_{\mathcal{C}^1(-\infty,r)}<+\infty$ and $\eta \psi=0$ in a neighborhood of $r$, from \eqref{eqn:var2} we deduce
\begin{multline}\label{eqn:var3}
\int_{C_{(-\infty,r)}}   \left(2 \sum_i (\pa_x u_i)^2  - \left(\sum_i |\nabla u_i|^2 + \sum_{i<j}u_i^2 u_j^2\right)\right) \eta \psi' \\
= \int_{C_{(-\infty,r)}}   \left( \sum_i |\nabla u_i|^2 + \sum_{i<j} u_i^2 u_j^2-2 \sum_i (\pa_x u_i)^2  \right) \eta'\psi.
\end{multline}
Denoting by
\[
\gamma= \left( \sum_i |\nabla u_i|^2 + \sum_{i<j} u_i^2 u_j^2  -2 \sum_i (\pa_x u_i)^2\right)\psi,
\]
the right hand side is
\begin{multline*}
\int_0^{k\pi} \left( \int_{r-\eps}^r \eta'(x) \gamma(s,y)\,\de x \right)\, \de y = -\int_0^{k\pi} \gamma(r-\eps,y)\,\de y \\
-  \int_0^{k\pi} \left( \int_{r-\eps}^r \eta(s) \pa_x \gamma(x,y)\,\de x \right)\, \de y \\
= \int_{\Sigma_r} \left(2 \sum_i (\pa_x u_i)^2 - \left( \sum_i |\nabla u_i|^2 + \sum_{i<j} u_i^2 u_j^2\right)\right) \psi + o(1)
\end{multline*}
as $\eps \to 0$, where the last identity follows from the regularity of $(u_1,\dots,u_k)$ and from the $\mathcal{C}^1$-boundedness of $\psi$ and $\eta$. Passing to the limit as $\eps \to 0$ in the \eqref{eqn:var3}, we deduce that for every $\psi \in \mathcal{C}^1((-\infty,r])$ such that $\|\psi\|_{\mathcal{C}^1((-\infty,r])} <+\infty$ it results
\begin{multline*}
\int_{C_{(-\infty,r)}}   \left(2 \sum_i (\pa_x u_i)^2  - \left(\sum_i |\nabla u_i|^2 + \sum_{i<j}u_i^2 u_j^2\right)\right) \psi' \\
= \int_{\Sigma_r} \left(2 \sum_i (\pa_x u_i)^2 - \left( \sum_i |\nabla u_i|^2 + \sum_{i<j} u_i^2 u_j^2\right)\right) \psi.
\end{multline*}
Choosing $\psi=1$ we obtain the thesis.
\end{proof}

This result permits to prove an Almgren monotonicity formula for a solution $(u_1,\dots,u_k)$ of \eqref{eqn: system k comp} in $C_{(-\infty,b)}$ such that \eqref{eqn:decay} holds. For such a solution, let us set
\[
E^{unb}(r):= \int_{C_{(-\infty,r)}} \sum_{i=1}^k |\nabla u_i|^2 + 2 \sum_{1 \le i<j\le k} u_i^2 u_j^2,
\]
We will briefly write $E$ in the rest of the subsection. Clearly, Lemma \ref{lem:decay to finite E} and the fact that $\mathcal{E}(r) \to 0$ as $r \to -\infty$ (see Remark \ref{decay energy}) implies that
\begin{equation}\label{eqn:finite energy 2}
E(r)<+\infty \quad \forall r<b \quad \text{and} \quad \lim_{r \to -\infty} E(r) =0.
\end{equation}
By regularity, $E, \mathcal{E}$ and $H$ are smooth. A direct computation shows that $E$ and $\mathcal{E}$ are increasing in $r$. As far as $H$ is concerned, with respect to the previous subsection we cannot deduce the \eqref{eqn:der H} by means of a simple integration by parts, because we are working in an unbounded domain. However,

\begin{lemma}
Let $(u_1,\ldots,u_k)$ be a solution to \eqref{eqn: system k comp} in $C_{(-\infty,b)}$, such that \eqref{eqn:decay} holds. Then
\[
H'(r) = 2 \int_{\Sigma_r} \sum_{i=1}^k u_i \pa_\nu u_i = 2 E(r)
\]
for every $r <b$. In particular, $H$ is nondecreasing.
\end{lemma}

\begin{proof}
For every $s<r <b$, the divergence theorem and the periodicity of $(u_1,\dots,u_k)$ imply that
\begin{equation}\label{eq2}
\begin{split}
E(r) &= E(s) + \int_{C_{(s,r)}}  \sum_i |\nabla u_i|^2+ 2 \sum_{i<j} u_i^2 u_j^2 \\
& = E(s) - \int_{\Sigma_s} \sum_i u_i \pa_x u_i + \int_{\Sigma_x} \sum_i u_i \pa_\nu u_i.
\end{split}
\end{equation}
We consider the second term on the right hand side. Let $\eta \in C^{\infty}_c(-1,1)$ be a non negative cut-off function, even with respect to $r = 0$, such that $\eta(0)=1$ and $\eta \le 1$ in $(-1,1)$. Let $\eta_s(x) = \eta(x-s)$; testing the equation \eqref{eqn:system k} with $u_i \eta_s$ in $C_{(s-1,s)}$, we find
\[
	\int_{C_{(s-1,s)}} \nabla u_i \cdot \nabla (u_i \eta_s) +  u_i^2 \sum_{i \neq j} u_j^2 \eta_s= \int_{\Sigma_s} u_i \partial_x u_i
\]
Summing up for $i =1,\dots,k$, we obtain
\begin{equation}\label{eq1}
\begin{split}
	\int_{\Sigma_s} \sum_i u_i \partial_x u_i  &= \int_{C_{(s-1,s)}} \sum_i \left(u_i \partial_x u_i \eta_s'+ |\nabla u_i|^2\eta_s \right) + 2\sum_{i < j} u_i^2 u_j^2 \eta_s \\
	&\leq C(\eta') \sum_i \|u_i\|^2_{H^1(C_{(s-1,s)})} + E(s),
\end{split}
\end{equation}
where the last estimate follows from the H\"older inequality. We claim that
\[
\sum_i \|u_i\|_{H^1(C_{(s-1,s)})} \to 0  \qquad \text{as $s \to -\infty$}.
\]
This is a consequence of the Poincar\'e inequality
\[
\int_{C_{(s-1,s)}} u^2 \le C \left( \int_{\Sigma_s} u^2 + \int_{C_{(s-1,s)}} |\nabla u|^2 \right) \qquad \forall u \in H^1(C_{(s-1,s)})
\]
together with assumption \eqref{eqn:decay} and the fact that $E(s) \to 0$ as $s \to -\infty$ (see \eqref{eqn:finite energy 2}). Thus, from the \eqref{eq1} we deduce that
\[
\lim_{s \to -\infty} \int_{\Sigma_s} \sum_i u_i \partial_{x} u_i = 0,
\]
which in turns can be used in the \eqref{eq2} to obtain the thesis:
\[
E(r)= \lim_{s \to -\infty} \left(E(s) -\int_{\Sigma_s} \sum_i u_i \pa_x u_i + \int_{\Sigma_x} \sum_i u_i \pa_\nu u_i\right) = \int_{\Sigma_x} \sum_i u_i \pa_\nu u_i. \qedhere
\]
\end{proof}

In light of the previous results, the proof of the following statements are straightforward modification of the proofs of Proposition \ref{prp:Almgren Neumann}, Corollary \ref{doubling} and Lemmas \ref{lem monotonicity for E} and \ref{cor estimate for N}.

\begin{proposition}\label{prp: Almgren k finite}
Let $(u_1,\ldots,u_k)$ be a solution of \eqref{eqn:system k} in $C_{(-\infty,b)}$ such that \eqref{eqn:decay} holds. The\emph{Almgren quotient}
\[
N^{unb}(r):= \frac{E^{unb}(r)}{H(r)}
\]
is well defined in $(-\infty,b)$ and nondecreasing. Moreover,
\[
\int_{-\infty}^r \frac{\int_{\Sigma_s} \sum_{i<j} u_i^2 u_j^2}{H(s)}\, \de s \le N(r).
\]
Analogously, the function $\displaystyle \mathfrak{N}^{unb}(r):= \frac{\mathcal{E}^{unb}(r)}{H(r)}$ is well defined in $(-\infty,b)$ and nondecreasing.
\end{proposition}

We will briefly write $N$ and $\mathfrak{N}$ instead of $N^{unb}$ and $\mathfrak{N}^{unb}$ in the rest of this subsection.

\begin{corollary}\label{doubling finite}
Let $(u_1,\dots, u_k)$ be a solution of \eqref{eqn:system k} in $C_{(-\infty,b)}$ such that \eqref{eqn:decay} holds.
\begin{itemize}
\item[($i$)] If $N(r) \ge \underline{d}$ for $r\ge s$, then
\[
\frac{H(r_1)}{e^{2 \underline{d} r_1}} \le \frac{H(r_2)}{e^{2 \underline{d} r_2}}  \qquad \forall \ s \le r_1<r_2<b,
\]
\item[$ii$)] If $N(r) \le \overline{d}$ for $r \le t<b$, then
\[
  \frac{H(r_1)}{e^{2 \overline{d} r_1}} \ge \frac{H(r_2)}{e^{2 \overline{d}r_2}} \qquad \forall \ r_1<r_2 \le t.
\]
\end{itemize}
\end{corollary}

For a fixed $r_0 < b$, let us introduce
\[
	\varphi(r;r_0):= \int_{r_0}^r \frac{\de s}{H(s)^{1/4}}.
\]
The function $\varphi$ is positive and increasing in $\R^+$; thanks to point ($i$) of Corollary \ref{doubling finite} and to the monotonicity of $N$, whenever $(u,v)$ is nontrivial $\varphi$ is bounded by a quantity depending only $H(r_0)$ and $N(r_0)$:
\begin{equation}\label{eqn:estimate for phi finite}
	\varphi(r;r_0) \leq 2\frac{e^{\frac{1}{2}N(r_0)r_0}}{H(r_0)^{\frac14} N(r_0)} \left[e^{-\frac12 N(r_0) r_0}-e^{-\frac12 N(r_0)r} \right].
\end{equation}
This, together with the monotonicity of $\varphi(\cdot;r_0)$, implies that if $b =+\infty$ then there exists the limit
\[
\lim_{r \to +\infty} \varphi(r;r_0)<+\infty.
\]

\begin{lemma}\label{lem monotonicity for E finite}
Let $(u_1,\dots,u_k)$ be a solution of \eqref{eqn: system R} in $C_{(-\infty,b)}$ such that \eqref{eqn:decay} hold. Let $r_0 \in (-\infty,b)$, and assume that
\begin{equation}\label{symmetry u-v finite}
u_{i+1}(x,y)=u_i(x,y-\pi) \quad \text{and} \quad u_1\left(x,\tau+y \right)= u_1\left(x,\tau -y\right)
\end{equation}
where $\tau \in [0,k\pi)$. There exists $C>0$ such that the function $\displaystyle{r \mapsto \frac{E(r)}{e^{2r}} e^{ C \varphi(r;r_0)}}$ is nondecreasing in $r$ for $r>r_0$.
\end{lemma}

\begin{lemma}\label{cor estimate for N finite}
Let $(u_1,\dots,u_k)$ be a nontrivial solution of \eqref{eqn:system k} in $C_{\infty}$,  and assume that \eqref{eqn:decay} and \eqref{symmetry u-v finite} hold. If $d:=\lim_{r \to +\infty} N(r)<+\infty$, then $d \ge 1$ and
\[
\lim_{r \to +\infty}\frac{E(r)}{e^{2r}}>0.
\]
\end{lemma}

\begin{remark}\label{rem:vale tutto sul beta problema}
The achievements of this section hold true for solutions to
\[
\begin{cases}
-\Delta u_i = - \beta u_i \sum_{j \neq i} u_j^2 \\
u_i>0
\end{cases}
\]
with the energy density
\[
\sum_i |\nabla u_i|^2 + 2\sum_{i<j} u_i^2 u_j^2 \quad \text{replaced by} \sum_i |\nabla u_i|^2 + 2\beta\sum_{i<j} u_i^2 u_j^2.
\]
\end{remark}

\subsection{Monotonicity formulae for harmonic functions}\label{sub:monot har finite}

Here we prove some monotonicity formulae for harmonic functions of the plane which are $2\pi$ periodic in one variable. In what follows, in the definition of $C_{(a,b)}$ and $\Sigma_r$ we mean $k=2$. The following results will come useful in section \ref{sec: blow-down}.

Firstly, it is not difficult to obtain the counterpart of Lemma \ref{lem: Pohozaev harmonic k}.

\begin{lemma}\label{lem: Pohozaev harmonic}
Let $\Psi$ be an entire harmonic function in $C_{(a,b)}$. Then the function
\[
    r \mapsto \int_{\Sigma_r} |\nabla \Psi|^2 - 2 \Psi_x^2
\]
is constant.
\end{lemma}
\begin{proof}
We proceed as in the proof of Lemma \ref{lem: Pohozaev harmonic k}: for $a<r_1<r_2<b$, we test the equation $-\Delta \Psi= 0$ with $\Psi_x$ in $C_{(r_1,r_2)}$ and integrate by parts.
\end{proof}

In what follows we consider a harmonic function $\Psi$ defined in an unbounded cylinder $C_{(-\infty,b)}$, with $b \in \R$ (the choice $b=+\infty$ is admissible). We assume that
\begin{equation}\label{eqn:decay har}
H(r;\Psi):= \int_{\Sigma_r} \Psi^2 \to 0 \quad \text{as $r \to -\infty$}.
\end{equation}

\begin{lemma}\label{lem: le armoniche vanno a zero}
Let $\Psi$ be a harmonic function in $C_{(-\infty,b)}$ such that \eqref{eqn:decay har} holds true.
Then
\begin{itemize}
    \item[($i$)] for every $r \in \R$ it results
    \(
	   \displaystyle E^{unb}(r;\Psi):= \int_{C_{(-\infty,r)}} |\nabla \Psi|^2  < +\infty
    \)
    \item[($ii$)] it results
    \begin{equation}\label{eqn: pohozaev per le arminiche che decadono}
        \int_{\Sigma_{r}} |\nabla \Psi|^2 = 2 \int_{\Sigma_{r}} (\partial_x \Psi)^2
    \end{equation}
\end{itemize}
\end{lemma}
\begin{proof}
In light of Lemma \ref{lem: Pohozaev harmonic}, it is not difficult to adapt the proof of Lemma \ref{decay energy} and obtain ($i$). As far as ($ii$), we can proceed as in the proof of Lemma \ref{lem:pohozaev energy} (note that, thanks to \eqref{eqn:decay har}, it results $H'(r;\Psi) = 2E^{unb}(r;\Psi)$).
\end{proof}

\begin{proposition}\label{thm: Almgren per armoniche}
Let $\Psi$ be a nontrivial harmonic function in $C_{(-\infty,b)}$, such that \eqref{eqn:decay har} holds true. The \emph{Almgren quotient}
\[
    N^{unb}(r;\Psi) := \frac{\int_{C_{(-\infty,r)}} |\nabla \Psi|^2}{\int_{\Sigma_{r}} \Psi^2}
\]
is nondecreasing in $r$. If $N(\cdot;\Psi)$ is constant for $r$ in some non empty open interval $(r_1,r_2)$, then $N(r;\Psi)$ is constant for all $r \in \R$ and there exists a positive integer $d \in \N$ such that $N(r;\Psi)=d$; furthermore,
\[
    \Psi(x,y) = \left[C_1 \cos(d y) + C_2 \sin(d y)\right]e^{d x}
\]
for some $C_1,C_2 \in \R$.
\end{proposition}
\begin{proof}
The Almgren quotient is well defined, thanks to Lemma \ref{lem: le armoniche vanno a zero}. To prove its monotonicity, we compute the logarithmic derivative by means of the Pohozaev identity \eqref{eqn: pohozaev per le arminiche che decadono}
\[
    \frac{(N^{unb})'(r;\Psi)}{N^{unb}(r;\Psi)} = \frac{\int_{\Sigma_{r}} |\nabla \Psi|^2}{\int_{C_{(-\infty,r)}} |\nabla \Psi|^2}-2\frac{\int_{\Sigma_{r}} \Psi\partial_x\Psi }{ \int_{\Sigma_{r}} \Psi^2} = 2 \frac{\int_{\Sigma_{r}} |\partial_x \Psi|^2}{\int_{\Sigma_{r}} \Psi \partial_x \Psi} - 2\frac{\int_{\Sigma_{r}} \Psi\partial_x\Psi }{ \int_{\Sigma_{r}} \Psi^2} \geq 0
\]
where in the last step we used the Cauchy-Schwarz inequality.

Let us assume now that $N^{unb}(r;\Psi)$ is constant for  $r \in (r_1,r_2)$. By the previous computations it follows that necessarily
\[
    \int_{\Sigma_{r}} |\partial_x \Psi|^2 \int_{\Sigma_{r}} \Psi^2 = \left(\int_{\Sigma_{r}} \Psi\partial_x\Psi \right)^2
\]
for every $r \in (r_1,r_2)$. Again from the Cauchy-Schwarz inequality, we evince that it must be
\[
   \partial_x \Psi = \lambda \Psi \qquad \text{on $\Sigma_r$}
\]
for some constant $\lambda \in \R$ and for every $r \in (r_1,r_2)$. Solving the differential equation, we find the $\Psi$ is of the form
\[
    \Psi(x,y) = \psi(y) e^{\lambda x}.
\]
This together with the equation $\Delta \Psi = 0$ yields,
\[
    \psi'' + \lambda^2 \psi = 0 \quad \Rightarrow \quad  \Psi(x,y) = \left[a \cos(\lambda y) + b \sin(\lambda y)\right] e^{\lambda x} \qquad \forall (x,y) \in (r_1,r_2) \times \R,
\]
and $\Psi$ can be uniquely extended to $\R^2$ by the unique continuation principle for harmonic functions. Since $\Psi$ satisfies the condition \eqref{eqn:decay har} and is nontrivial, it follows that $\lambda > 0$. The proof is complete, recalling the periodicity in $y$ of the function $\Psi$ and computing its Almgren quotient.
\end{proof}

\section{Proof of Theorem \ref{main theorem 1}}\label{sec:proof1}

In this section we construct a solution to \eqref{eqn: system R} modeled on the harmonic function $\Phi(x,y)= \cosh x \sin y $.

\subsection{Existence in bounded cylinders}\label{subs: existence bdd}

For every $R>0$ we construct a solution $(u_R,v_R)$ to
\begin{subequations}\label{system in C_R}
\begin{equation}\label{eqn: system in C_R cosh}
\begin{cases}
-\Delta u=-u v^2 & \text{in $C_R$} \\
-\Delta v=-u^2 v & \text{in $C_R$} \\
u,v>0
\end{cases}
\end{equation}
(equivalently, we can consider the problem in $(-R,R) \times (0,2\pi)$ with periodic boundary condition on the sides $[-R,R] \times \{0,2\pi\}$) with Dirichlet boundary condition
\begin{equation}\label{eqn: bc system in C_R cosh}
u=\Phi^+, \quad v=\Phi^- \quad \text{on $\Sigma_R \cup \Sigma_{-R}$},
\end{equation}
\end{subequations}
and exhibiting the same symmetries of $(\Phi^+,\Phi^-)$. To be precise:

\begin{proposition}\label{existence thm in bdd cylinders}
There exists a solution $(u_R,v_R)$ to problem \eqref{eqn: system in C_R cosh} with the prescribed boundary conditions \eqref{eqn: bc system in C_R cosh}, such that
\begin{itemize}
\item[1)] $u_R(-x,y)=u_R(x,y)$ and $v_R(-x,y)=v_R(x,y)$,
\item[2)] the symmetries
\begin{align*}
v_R(x,y) = u_R(x,y-\pi) \quad  & \quad u_R(\pi- x,y)= v_R(\pi +x,y)\\
u_R\left(x, \frac{\pi}{2}+y\right)  =  u_R\left(x,\frac{\pi}{2}-y\right) \quad & \quad v_R\left(x,\frac{3}{2}\pi+y\right)= v_R\left(x,\frac{3}{2}\pi-y\right)
\end{align*}
hold,
\item[3)] $u_R-v_R \ge 0$ in $\{\Phi>0\}$ and $v_R-u_R \ge 0$ in $\{\Phi<0\}$,
\item[4)] $u_R > \Phi^+$ and $v_R > \Phi^-$.
\end{itemize}
\end{proposition}

\begin{remark}\label{rem su Neu}
In light of the eveness of $(u_R,v_R)$ in $x$, it results
\[
\pa_x u = 0 = \pa_x v \qquad \text{on $\Sigma_0$}.
\]
As a consequence, the monotonicity formulae proved in subsection \ref{sub:monot Neumann} hold true for $(u_R,v_R)$ in the semi-cylinder $C_{(0,R)}$.
\end{remark}

In order to keep the notation as simple as possible, in what follows we will refer to a solution of \eqref{eqn: system in C_R cosh}-\eqref{eqn: bc system in C_R cosh} as to a solution of \eqref{system in C_R}.

\begin{proof}
Let
\[\mathcal{U}_R:= \left\{ (u,v) \in (H^1(C_R))^2 \left|
	\begin{array}{l}
		u=\Phi^+,\, v=\Phi^- \text{ on } \Sigma_R \cup \Sigma_{-R}, \, u \ge 0, \\
		\text{$u-v \ge 0$ in $\{\Phi \ge 0\}$,} \\
		v(x,y)=u(x,y-\pi), \, u(-x,y)=u(x,y), \\
		u(x,\pi-y)=v(x,\pi+y),  \, u\left(x,\frac{\pi}{2}+y\right)= u\left(x,\frac{\pi}{2}-y\right)
	\end{array} \right. \right\}.
\]
Note that if $(u,v) \in \mathcal{U}_R$ then $v$ is nonnegative, even in $x$ and symmetric in $y$ with respect to $\frac{3}{2}\pi$; moreover, $u-v \le 0$ in $\{\Phi<0\}$. It is immediate to check that $\mathcal{U}_R$ is weakly closed with respect to the $H^1$ topology. %: indeed, $\mathcal{U}_R$ is closed (in the strong topology) and convex. \\
We seek solutions of \eqref{system in C_R} as minimizers of the energy functional
%\begin{equation}\label{eqn: energy function}
\[
J(u,v):= \int_{C_R}  |\nabla u|^2+ |\nabla v|^2 + u^2 v^2
\]
%\end{equation}
in $\mathcal{U}_R$. The existence of at least one minimizer is given by the direct method of the calculus of variations; for the coercivity of the functional $J$, we use the following Poincar\'e inequality:
\begin{equation}\label{Poincarè}
\int_{C_R} u^2 \le C \left( \int_{\Sigma_{-R}} u^2 + \int_{C_R} |\nabla u|^2 \right) \qquad \forall u \in H^1(C_R),
\end{equation}
where $C$ depends only on $R$. To show that a minimizer satisfies equation \eqref{system in C_R}, we consider the parabolic problem
\begin{equation}\label{parabolic system}
\begin{cases}
U_t- \Delta U=-U V^2 & \text{in $(0,+\infty) \times C_R$} \\
V_t- \Delta V= -U^2 V & \text{in $(0,+\infty) \times C_R$} \\
U=\Phi^+, \ V= \Phi^- & \text{on $(0,+\infty) \times (\Sigma_R \cup \Sigma_{-R})$} \\
\end{cases}
\end{equation}
with initial condition in $\mathcal{U}_R$. There exists a unique local solution $(U,V)$; by Lemma \ref{lem:parabolic_max_principle} if follows $U,V \ge 0$; hence, the maximum principle gives
\[
0\le U \le \sup_{C_R} \Phi^+ \quad \text{and} \quad 0\le V \le \sup_{C_R} \Phi^-.
\]
This control reveals that $(U,V)$ can be uniquely extended in the whole $(0,+\infty)$. Since
\begin{equation}\label{dissipative estimate}
\frac{\de}{\de t} J(U(t,\cdot),V(t,\cdot))=- 2\int_{C_R} \left(U_t^2+V_t^2\right) \le 0,
\end{equation}
that is, the energy is a Lyapunov functional, from the parabolic theory it follows that for every sequence $t_i \to +\infty$ there exists a subsequence $(t_j)$ such that $(U(t_j\cdot), V(t_j,\cdot))$ converges to a solution $(u,v)$ of \eqref{system in C_R}. Therefore, in order to prove that $(u_R,v_R)$ solves \eqref{system in C_R}, it is sufficient to show that there exists an initial condition in $\mathcal{U}_R$ such that the limiting profile $(u,v)$ coincides with $(u_R,v_R)$. We use the fact that
\begin{equation}\label{claim 2}
\text{$\mathcal{U}_R$ is positively invariant under the parabolic flow}.
\end{equation}
To prove this claim, we firstly note that by the symmetry of initial and boundary conditions and by the uniqueness of the solution to problem \eqref{parabolic system}, we have
\begin{equation}\label{eq4}
\begin{split}
V(t,x,y)= U(t,x,y-\pi), \quad &\quad U(t,-x,y)=U(t,x,y), \\
V(t,x,\pi + y)= U(t,x,\pi-y), \quad &\quad
 U\left(t,x,\frac{\pi}{2}+y\right)=U\left(t,x,\frac{\pi}{2}-y\right).
\end{split}
\end{equation}
This implies
\[
U(t,x,\pi)-V(t,x,\pi)=0 \qquad \forall (t,x) \in (0,+\infty) \times [-R,R].
\]
Furthermore, using the \eqref{eq4} and the periodicity of $(U,V)$
\begin{align*}
U(t,x,0)- V(t,x,0) = U(t,x,0)-V(t,x,2\pi)=0 \qquad &\forall (t,x) \in (0,+\infty) \times [-R,R] \\
U(t,x,2\pi)- V(t,x,2\pi) = U(t,x,2\pi)-V(t,x,0)=0 \qquad &\forall (t,x) \in (0,+\infty) \times [-R,R].
\end{align*}
This means that $U-V=0$ on $\{\Phi=0\}$. Let us introduce $D_R:= \{\Phi>0\} \cap C_R$. For every $(u_0,v_0) \in \mathcal{U}_R$, we have
\begin{equation}\label{system 4.5}
\begin{cases}
(U-V)_t - \Delta (U-V)= UV (U-V) & \text{in $(0,+\infty) \times D_R$} \\
U-V \ge 0 & \text{on $\{0\} \times D_R$} \\
U-V \ge 0 & \text{on $[0,+\infty) \times \pa D_R$}.
\end{cases}
\end{equation}
Lemma \ref{lem:parabolic_max_principle} implies $U-V \ge 0$ in $(0,+\infty) \times D_R$. This completes the proof of the claim.

Let us consider the equation \eqref{parabolic system} with the initial conditions $U(0,x,y)=u_R(x,y)$, $V(0,x,y)=v_R(x,y)$; let us denote $(U^R,V^R)$ the corresponding solution. On one side, by minimality,
\begin{equation*}%\label{minim (u_R,v_R)}
J(u_R,v_R) \le J(U^R(t,\cdot), V^R(t,\cdot)) \qquad \forall t \in (0,+\infty);
\end{equation*}
we point out that this comparison is possible because of \eqref{claim 2}. On the other side, by \eqref{dissipative estimate},
\begin{equation*}%\label{cons dissipative}
J(U^R(t,\cdot), V^R(t,\cdot)) \le J(u_R,v_R) \qquad \forall t \in (0,+\infty).
\end{equation*}
%Comparing \eqref{minim (u_R,v_R)} and \eqref{cons dissipative}
We deduce that $J(U^R,V^R)$ is constant, which in turns implies (use again \eqref{dissipative estimate}),
\[
U_t^R(t,x,y) = V_t^R(t,x,y) \equiv 0 \quad \Rightarrow \quad  U^R(t,x,y) = u_R(x,y), \quad V^R(t,x,y)=v_R(x,y).
\]
By the above argument, as $(u_R,v_R)$ coincides with the asymptotic profile of a solution of the parabolic problem \eqref{parabolic system}, it solves \eqref{system in C_R}. Points 1)-3) of the thesis are satisfied due to the positive invariance of $\mathcal{U}_R$.  The strong maximum principle yields $u_R>0$ and $v_R>0$. Moreover,
\[
\begin{cases}
-\Delta(u_R-v_R-\Phi)= u_R v_R (u_R-v_R) \ge 0 & \text{in $D_R$} \\
u_R-v_R -\Phi = 0& \text{on $\pa D_R$}
\end{cases} \quad \Rightarrow \quad u_R-v_R-\Phi \ge 0  \quad \text{in $D_R$},
\]
so that by the strong maximum principle and the fact that $u_R, v_R>0$ we deduce $u_R > \Phi^+$. Analogously, $v_R > \Phi^-$.
\end{proof}

\begin{remark}
The existence of a positive solution satisfying the conditions 1)-2) of the Proposition can be proved by means of the celebrated Palais' Principle of Symmetric Criticality. To do this, it is sufficient to minimize the functional $J$ in the weakly closed set
\[\mathcal{V}_R:= \left\{ (u,v) \in (H^1(C_R))^2 \left|
	\begin{array}{l}
		u=\Phi^+,\, v=\Phi^- \text{ on } \Sigma_R \cup \Sigma_{-R},  \\
		v(x,y)=u(x,y-\pi), \, u(-x,y)=u(x,y),  \\
		u(x,\pi-y)=v(x,\pi+y), \,
		u\left(x,\frac{\pi}{2}+y\right)= u\left(x,\frac{\pi}{2}-y\right)
	\end{array} \right. \right\},
\]
and apply the maximum principle. We choose a more complicated proof since we will strongly use the pointwise estimates given by point 4).
\end{remark}

\subsection{Compactness of the family $\{(u_R,v_R)\}$}\label{sub:compat 1}

In this section we aim at proving that, up to a subsequence, the family $\{(u_R,v_R): R>1\}$ obtained in Proposition \ref{existence thm in bdd cylinders} converges, as $R \to +\infty$, to a solution $(u,v)$ of \eqref{eqn: system R} defined in the whole $C_\infty$. Then, by looking at $(u,v)$ as defined in $\R^2$ (this is possible thanks to the periodicity), we obtain a solution of \eqref{eqn: system R} satisfying the conditions 1)-5) of Theorem \ref{main theorem 1}. At a later stage, we will also obtain the estimates of points 6) and 7).

\medskip

We denote $E_R, \mathcal{E}_R, H_R, N_R$ and $\mathfrak{N}_R$ the functions $E^{sym},H, \mathcal{E}^{sym}, N^{sym}$ and $\mathfrak{N}^{sym}$ (which have been defined in subsection \ref{sub:monot Neumann}) when referred to $(u_R,v_R)$. As observed in Remark \ref{rem su Neu}, for these quantities the results of subsection \ref{sub:monot Neumann} apply.

We will obtain compactness of the sequence $(u_R,v_R)$ using some uniform-in-$R$ control on $N_R$ and $H_R$. We start with a uniform (in both $r$ and $R$) upper bound for the Almgren quotients $N_R(r)$.

\begin{lemma}\label{lem: N less 2}
There holds $N_R(r) \le 2$, for every $R>0$ and $r \in (0,R)$.
\end{lemma}

\begin{proof}
It is an easy consequence of the monotonicity of $N_R$ and of the minimality of $(u_R,v_R)$ for the functional $J$ in $\mathcal{U}_R$: noting that $J(u_R,v_R)= \mathcal{E}_R(R)$, we compute
\[
N_R(r) \le N_R(R) \le \frac{2 \mathcal{E}_R(R)}{H_R(R)} \le \frac{1}{\int_{\Sigma_R} \Phi^2} \int_{C_{(0,R)}} |\nabla \Phi|^2= 2\tanh R.
\]
We used the fact that the restriction of $(\Phi^+,\Phi^-)$ in $C_R$ is an element of $\mathcal{U}_R$ for every $R$, and the boundary condition of $(u_R,v_R)$ on $\Sigma_R$.
\end{proof}

In the proof of the following Lemma we will exploited the compactness of the local trace operator $T_{\Sigma_1}: u \in H^1(C_{(0,1)}) \mapsto u|_{\Sigma_1} \in L^2(\Sigma_1)$, see Corollary \ref{cor:local compactenss traces}.

\begin{lemma}\label{lem:H(1) limitato}
There exists $C>0$ such that $H_R(1) \le C$ for every $R>1$.
\end{lemma}
\begin{proof}
By contradiction, assume that $H_{R_n}(1) \to +\infty$ for a sequence $R_n \to +\infty$. %Since $H_R$ is increasing, it is clear that $H_{R_n}(r) \to \infty$ for every $r>1$.
Let us introduce the sequence of scaled functions
\[
	(\hat u_n(x,y),\hat v_n(x,y)) := \frac{1}{\sqrt{H_{R_n}(1)}}\left(u_{R_n}(x,y),v_{R_n}(x,y)\right).
\]
We wish to prove a convergence result for such a sequence, in order to obtain a uniform lower bound for $N_{R_n}(1)$. In a natural way, the scaling leads us to consider, for $r \in (0,1)$, the quantities
\begin{gather*}
    \hat E_n(r) := \int_{C_{(0,r)}}  |\nabla \hat u_{n}|^2 +|\nabla \hat v_n|^2+ 2 H_{R_n}(1) \hat u_{n}^2\hat v_{n}^2,\\
    \hat H_n(r) := \int_{\Sigma_{r}}  \hat u_{n}^2+ \hat v_n^2, \quad \hat N_n(r) :=  \frac{\hat E_n(r)}{\hat H_n(r)}.
\end{gather*}
By construction, it holds $\hat H_n(1) = 1$ and $\hat N_n(r) = N_{R_n}(r) \leq 2$; therefore, thanks to Lemma \ref{lem: N less 2}
\begin{equation}\label{eqn: uniform bound E_n}
        \int_{C_{(0,1)}} |\nabla \hat u_{n}|^2+|\nabla \hat v_{n}|^2  \leq \hat E_n(1) = \hat N_n(1) \hat H_n(1) \leq 2 \qquad \forall r \in (0,1],
\end{equation}
which gives a uniform bound in the $H^1(C_{(0,1)})$ norm of the sequence $(\hat u_n,\hat v_n)$ (we can use a Poincar\'e inequality of type \eqref{Poincarè}). Then, we can extract a subsequence which converges weakly in $H^1(C_{(0,1)})$ to some limiting profile $(\hat u, \hat v)$, which is nontrivial in light of the compactness of the local trace operator $T_{\Sigma_1}$ and of the fact that $\hat H_n(1)=1$. Since the set of the restrictions to $C_{(0,1)}$ of functions of $\mathcal{U}_R$ is closed in the weak $H^1(C_{(0,1)})$ topology, $\hat u$ and $\hat v$ are nonnegative functions with the same symmetries of $(u_R,v_R)$; moreover we can show that $(\hat u, \hat v)$ satisfies the segregation condition $\hat u \hat v =0$ a.e. in $C_{(0,1)}$. Indeed, by the compactness of the Sobolev embedding $H^1(C_{(0,1)}) \hookrightarrow L^4(C_{(0,1)})$ we deduce that the interaction term
\[
	I(u,v) :=\int_{C_{(0,1)}} u^2 v^2
\]
is continuous in the weak topology of $(H^1(C_{(0,1)}))^2$. From the estimate \eqref{eqn: uniform bound E_n}, we infer
\[
	2 H_{R_n}(1) I( \hat u_n, \hat v_n) \leq \hat E_n(1) \leq 2;
\]
passing to the limit as $n \to +\infty$, we conclude
\[
	I(\hat u, \hat v) = \lim_{n \to \infty} I(\hat u_n, \hat v_n) = 0 \quad \Rightarrow \quad  \hat u \hat v =0 \text{ a.e. in } C_{(0,1)}.
\]
Moreover, from the compactness of the local trace operator $T_{\Sigma_1}$, we also deduce $\int_{\Sigma_1} \hat u^2 + \hat v^2 = 1$. Let us consider the functional
\[
	J^{\infty}(u,v) := \int_{C_{(0,1)}} |\nabla u|^2 + |\nabla v|^2,
\]
defined in the set
\[
	\mathcal{M}:= \left\{ (u,v) \in (H^1(C_{(0,1)}))^2 \left|
	\begin{array}{l}
		\int_{\Sigma_1} u^2+v^2 = 1, \\
		v(x,y)=u(x,y-\pi), \ u v = 0 \text{ a.e. in } C_1
	\end{array} \right.\right\}.
\]
Due to the compactness of the trace operator, one can check that $\mathcal{M}$ is closed in the weak $(H^1(C_{(0,1)}))^2$ topology. It is clear that $(\hat u,\hat v) \in \mathcal{M}$. We claim that
\[
	\inf_{(u,v) \in \mathcal{M}}  J^{\infty}(u,v) =: m > 0.
\]
Indeed, let us assume by contradiction that the infimum is 0: since the set $\mathcal{M}$ is weakly closed and $J^\infty$ is weakly lower semi-continuous and coercive, there exists $(\bar u, \bar v)$ such that	$J^{\infty}(\bar u,\bar v) = 0$. It follows that $(\bar u,\bar v)$ is a vector of constant functions; the symmetry and the segregation condition imply that $(\bar u,\bar v) \equiv(0,0)$, but this is in contrast with the fact that $(\bar u,\bar v) \in \mathcal{M}$. Thus, the weak convergence of the sequence $(\hat u_n,\hat v_n)$ entails
\[
	\liminf_{n \to \infty} \hat N_n(1) \ge \liminf_{n\to \infty} \int_{C_{(0,1)}} |\nabla \hat u_{n}|^2+ |\nabla \hat v_n|^2 \geq m > 0,
\]
so that whenever $n$ is sufficiently large
\begin{equation}\label{eqn:lower_estimate_N_n}
	N_{R_n}(1) = \hat N_n(1) \geq \frac12 m
\end{equation}
Thanks to Lemma \ref{lem: N less 2} we know that $\frac12 m \le N_{R_n}(1) \le 2$, and from the assumption $H_{R_n}(1) \to +\infty$ we deduce that (recall the \eqref{eqn:estimate for phi})
\begin{align*}
	\varphi_{R_n}(r;1): & = \int_1^r \frac{\de s}{H_{R_n}(s)^{1/4}} \\
	& \leq 2\frac{e^{\frac12 N_{R_n}(1)}}{H_{R_n}(1)^{\frac14} N_{R_n}(1)} \left[e^{-\frac12 N_{R_n}(1)}-e^{-\frac12 N_{R_n}(1)r} \right] \to 0
\end{align*}
as $n \to \infty$, for every $r>1$. In particular, there exists $C>0$ such that
\begin{equation}\label{fi bounded}
\varphi_{R_n}(r;1) \le C \qquad \forall 1 \le r \le R_n, \ \forall n.
\end{equation}
This implies that the sequence $(E_{R_n}(1))_n$ is bounded. To see this, we firstly note that $(u_{R_n},v_{R_n})$ satisfies the symmetry condition \eqref{symmetry u-v} which is necessary to apply Lemma \ref{lem monotonicity for E}; consequently, the variational characterization of $(u_{R_n},v_{R_n})$ (see also the proof of Lemma \ref{lem: N less 2} and the \eqref{fi bounded}) implies that
\begin{align*}
\frac{E_{R_n}(1)}{e^{2}} & \le e^{C\varphi_{R_n}(R_n;1)} \frac{E_{R_n}(R_n)}{e^{2R_n}} \le 2C \frac{\mathcal{E}_{R_n}(R_n)}{e^{2{R_n}}} \\
& \le C\frac{\int_{C_{(0,R_n)}}|\nabla \Phi|^2}{e^{2R_n}}= C \frac{\sinh R_n \cosh R_n  }{e^{2R_n}} \le C,
\end{align*}
where $C$ does not depend on $n$. Since $(E_{R_n}(1))_n$ is bounded and $(H_{R_n}(1))_n$ tends to infinity, we obtain
\[
	\lim_{n \to \infty} N_{R_n}(1) = \lim_{n \to \infty} \frac{E_{R_n}(1)}{H_{R_n}(1)} = 0,
\]
in contradiction with the \eqref{eqn:lower_estimate_N_n}
\end{proof}

\begin{proposition}\label{compactenss}
There exists a subsequence of $(u_R,v_R)$ which converges in $\mathcal{C}^2_{loc}(C_\infty)$, as $R \to +\infty$, to a solution $(u,v)$ of \eqref{eqn: system R} in the whole $C_\infty$. This solution satisfies point 2)-5) of Theorem \ref{main theorem 1}, and its Almgren quotient $N$ is such that
\[
N(r) \le 2 \quad \forall r>0 \quad \text{and} \quad \lim_{r \to +\infty} N(r) \ge 1.
\]
\end{proposition}
\begin{proof}
As $H_R(1)$ is bounded in $R$ and $N_R(1)\le 2$, also $E_R(1)$ is bounded in $R$. By means of a Poincar\'e inequality of type \eqref{Poincarè}, this induces a uniform-in-$R$ bound for the $H^1(C_{(0,1)})$ norm of $(u_R,v_R)$, which in turns, by the compactness of the trace operator, gives a uniform-in-$R$ bound for the $L^2(\pa C_{(0,1)})$ norm. Due to the subharmonicity of $(u_R,v_R)$, the $L^2(\pa C_{(0,1)})$ bound provides a uniform-in-$R$ bound for the $L^\infty$ norm of $(u_R,v_R)$ in every compact subset of $C_{(0,1)}$; the regularity theory for elliptic equations (see \cite{GiTr}) ensures that, up to a subsequence, $(u_R,v_R)$ converges in $\mathcal{C}^2_{loc}(C_{(0,1)})$, as $R \to +\infty$, to a solution $(u^1,v^1)$ of \eqref{eqn: system R} in $C_{(0,1)}$. As each $(u_R,v_R)$ is even in $x$, this solution can be extended by even symmetry in $x$ to $C_1$, and here satisfies the conditions 1)-4) of Proposition \ref{existence thm in bdd cylinders} (hence both $u^1$ and $v^1$ are nontrivial). The previous argument can be iterated: indeed, by Corollary \ref{doubling} and Lemma \ref{lem: N less 2}, we deduce
\[
H_R(r) \le \frac{H_R(1)}{e^{4}}e^{4 r} \le C e^{4r} \qquad \forall r>1;
\]
that is, a uniform-in-$R$ bound for $H_R(1)$ implies a uniform-in-$R$ bound for $H_R(r)$ for every $r > 1$. As a consequence we obtain, for every $r> 1$, a solution $(u^{r},v^{r})$ to equation \eqref{eqn: system R} in $C_{r}$. A diagonal selection gives the existence of a solution $(u,v)$ to \eqref{eqn: system R} in the whole $C_\infty$. This solution inherits by $(u^r,v^r)$ the conditions 1)-4) of Proposition \ref{existence thm in bdd cylinders}, and thanks to the $\mathcal{C}^2_{loc}(C_\infty)$ convergence and Lemma \ref{lem: N less 2} there holds
\[
N(r) = \frac{\int_{C_{(0,r)}} |\nabla u|^2 + |\nabla v|^2 + 2 u^2 v^2}{\int_{\Sigma_r} u^2+v^2} \le 2   \qquad \forall r >0.
\]
From Lemma \ref{cor estimate for N}, which we can apply in light of the symmetries of $(u,v)$, we conclude
\[
\lim_{r \to +\infty}  N(r) \ge 1. \qedhere
\]
\end{proof}

The following Lemma completes the proof of point 6) of Theorem \ref{main theorem 1}. After that, by means of the pointwise estimates $u>\Phi^+$ and $v > \Phi^-$ and Corollary \ref{doubling}, it is straightforward to obtain also point 7).

\begin{lemma}\label{lem: N less 1}
There holds \( \displaystyle l:=\lim_{r \to \infty} N(r) = 1.\)
\end{lemma}
\begin{proof}
In light of the fact that $l \ge 1$, it is sufficient to show that $ l \le 1$. Let $(u_{R_n},v_{R_n})$ be the convergent subsequence found in Proposition \ref{compactenss}, which we will simply denote $\{(u_n,v_n)\}$.
%It is convenient to introduce two auxiliary functions which we have already encountered in our analysis:
For $r>0$ we let
\[
	f_n(r) := \frac{\int_{C_{(0,r)}} u_n^2 v_n^2 }{H_{R_n}(r)}, \quad g_n(r):= \frac{\int_{\Sigma_{r}} u_n^2 v_n^2}{H_{R_n}(r)}.
\]
With $f$ and $g$ we identify the same quantities computed for the limiting profile $(u,v)$. Observe that $f_n,g_n,f$ and $g$ are continuous and nonnegative. By definition,
\begin{equation}\label{eqn: estim on f_R}
	f_n(r) \leq \frac{1}{2}N_{R_n}(r) \leq 1 \qquad \forall r > 0
\end{equation}
where we used Lemma \ref{lem: N less 2}. The uniform convergence of $(u_n,v_n)$ implies that $f_n \to f$ and $g_n \to g$ uniformly on compact intervals, while by Theorem \ref{prp:Almgren Neumann} we have
\[
	\int_0^r g_n(s) \, \mathrm{d} s \leq N_{R_n}(r) \quad \text{and} \quad \int_0^r g(s) \, \mathrm{d} s \leq N(r),
\]
so that in particular $g_n \in L^1(0,R)$ and $g \in L^1(\R^+)$. By means of the monotonicity formula for the Almgren quotient $\mathfrak{N}$, Proposition \ref{prp:Almgren Neumann}, it is possible to refine the computation in Lemma \ref{lem: N less 2}:
\[
N_{R_n}(r) = \mathfrak{N}_{R_n}(r) + f_{R_n}(r) \le \mathfrak{N}_{R_n}(R_n) + f_{R_n}(r) \le 1 + f_{n}(r).
\]
In light of the strong $H^1_{loc}(C_\infty)$ convergence of $(u_n, v_n)$ to $(u,v)$, we deduce
\[
N(r) \le 1 +\lim_{n \to +\infty} f_n(r) = 1 + f(r).
\]
We have to show that $f(r) \to 0$ as $r \to +\infty$. To prove this, we begin by computing the logarithmic derivative of $f_n$:
\[
	\frac{f_n'(r)}{f_n(r)} = \frac{\int_{\Sigma_{r}} u_n^2 v_n^2 }{ \int_{C_{(0,r)}} u_n^2 v_n^2 } - 2 \frac{E_{R_n}(r)}{H_{R_n}(r)} = \frac{g_n(r)}{f_n(r)} - 2N_{R_n}(r),
\]
where we used the fact that $H'_{R_n}(r)= 2 E_{R_n}(r)$, see the \eqref{eqn:der H}. Exploiting the strong $H^1$ convergence of the sequence $\{(u_n,v_n)\}$ and the fact that $\lim_{r \to +\infty}N(r) \ge 1$, we deduce that there exist $r_0,\delta>0$ such that $N_{R_n}(r_0) > \delta$ for every $n$ sufficiently large. Consequently, $f_n$ satisfies the inequality
\[
	f'_n(r) + 2 \delta f_n(r) \leq g_n(r) \qquad \text{for } r\in(r_0,R_n).
\]
Multiplying for $e^{2\delta r}$ and integrating in $(r_1,r_2)$ for $r_0<r_1<r_2<R_n$, we obtain
\[
	f_n(r_2) \leq e^{2\delta (r_1-r_2)} f_n(r_1) + \int_{r_1}^{r_2} g_n(s) e^{2\delta (s-r_2)} \,\mathrm{d} s \leq e^{2\delta (r_1-r_2)} + \int_{r_1}^{r_2} g_n(s) \,\mathrm{d}s,
\]
where we used the estimate \eqref{eqn: estim on f_R}. This implies
\[
	f(r_2) \leq e^{2\delta (r_1-r_2)} + \int_{r_1}^{r_2} g(s) \, \mathrm{d}s	\qquad \text{for $r_0<r_1<r_2$}.
\]
Since $g \in L^1(\R^+)$ and $f \ge 0$, choosing $r_1 = \frac12 r_2$ we find
\[
	\limsup_{r \to +\infty} f(r) = 0 = \lim_{r \to +\infty} f(r). \qedhere
\]
\end{proof}

\section{Proof of Theorem \ref{thm: main thm 2}}\label{sec:sec2}

In this section we construct a solution to \eqref{eqn: system R} modeled on the harmonic function $\Gamma(x,y)= e^x \sin y $. Our construction is based on the trivial observation that
\[
	\Phi_R(x,y) := 2\cosh (x+R) e^{-R} \sin y \to \Gamma(x,y) \quad \text{ as $R \to +\infty.$}
\]

\subsection{Existence in bounded cylinders}

As a first step, using the same line of reasoning developed in Proposition \ref{existence thm in bdd cylinders}, it is possible to show the existence of solution to the system
\begin{subequations}\label{system in C_R exp}
\begin{equation}\label{eqn: system in C_R exp}
	\begin{cases}
		-\Delta u=-u v^2 & \text{in $C_{(-3R,R)}$} \\
		-\Delta v=-u^2 v & \text{in $C_{(-3R,R)}$} \\
		u,v>0
	\end{cases}
\end{equation}
(equivalently, we can consider the problem in the rectangle $(-3R,R) \times (0,2\pi)$ with periodic boundary condition on the sides $[-3R,R] \times \{0,2\pi\}$) and such that
\begin{equation}\label{eqn: bc for system in C_R exp}
		u_R = \Phi^+_R, \quad  v_R = \Phi^-_R  \quad \text{on $\Sigma_{R} \cup \Sigma_{-3R}$}
\end{equation}
\end{subequations}
More precisely:

\begin{proposition}\label{existence thm in bdd cylinders exp}
There exists a solution $(u_R,v_R)$ to problem \eqref{eqn: system in C_R exp} with the prescribed boundary conditions \eqref{eqn: bc for system in C_R exp}, such that
\begin{itemize}
\item[1)] $u_R(-R-x,y)=u_R(-R+x,y)$ and $v_R(-R-x,y)=v_R(-R+x,y)$,
\item[2)] the symmetries
\begin{align*}
v_R(x,y) = u_R(x,y-\pi)\quad  & \quad u_R(x,\pi- y)= v_R(x,\pi +y)\\
u_R\left(x,\frac{\pi}{2}+y\right)  =  u_R\left(x,\frac{\pi}{2}-y\right) \quad & \quad v_R\left(x,\frac{3}{2}\pi + y\right)= v_R\left(x, \frac{3}{2}\pi-y\right)
\end{align*}
hold,
\item[3)] $u_R-v_R \ge 0$ in $\{\Phi_R>0\}$ and $v_R-u_R \ge 0$ in $\{\Phi_R<0\}$,
\item[4)] $u_R > (\Phi_R)^+$ and $v_R > (\Phi_R)^-$.
\end{itemize}
\end{proposition}
\begin{proof}[Sketch of proof]
One can recast the proof of Proposition \ref{existence thm in bdd cylinders} in this setting.
\end{proof}

\begin{remark}\label{rem su Neu 2}
In light of point 1) of the Proposition, it results
\[
\pa_x u_R = 0 = \pa_x v_R \qquad \text{on $\Sigma_{-R}$}.
\]
Therefore, the monotonicty formulae proved in subsection \ref{sub:monot Neumann} hold true for $(u_R,v_R)$ in the semi-cylinder $C_R$.
\end{remark}

\subsection{Compactness of the family $\{(u_R,v_R)\}$}\label{sub:compat 2}

As in the previous section, we denote as $E_R, \mathcal{E}_R, N_R$ and $\mathfrak{N}_R$ the functions $E^{sym}, \mathcal{E}^{sym}, N^{sym}$ and $\mathfrak{N}^{sym}$ defined in subsection \ref{sub:monot Neumann} when referred to $(u_R,v_R)$. We follow here the same line of reasoning adopted in subsection \ref{sub:compat 1}. Firstly, it is not difficult to modify the proof of Lemmas \ref{lem: N less 2} and \ref{lem:H(1) limitato} obtaining the following estimates:

\begin{lemma}\label{lem:N le 2 2}
There holds $N_R(r) \le 2$, for every $R>0$ and $r \in (-R,R)$.
\end{lemma}

\begin{lemma}\label{lem:H(1) limitato 2}
There exists $C>0$ such that $H_R(1) \le C$ for every $R>1$.
\end{lemma}

We are in position to show that the family $\{(u_R,v_R)\}$ is compact, in the following sense.

\begin{proposition}\label{compactness exp}
There exists a subsequence of $\{(u_R,v_R)\}$ which converges in $\mathcal{C}^2_{loc}(C_\infty)$, as $R \to +\infty$, to a solution $(u,v)$ of \eqref{eqn: system R} in the whole $C_\infty$. This solution has the properties 2)-4) of Proposition \ref{existence thm in bdd cylinders exp}.
\end{proposition}
\begin{proof}
As $H_{R}(1)$ is bounded in $R$ and $N_R(1) \le 2$, also $E_R(1)$ is bounded in $R$, and a fortiori
\[
\int_{C_1} |\nabla u_R|^2 +|\nabla v_R|^2 \le C \qquad \forall R>1.
\]
This estimate, the boundedness of $H_R(1)$ and a Poincar\`e inequality of type \eqref{Poincarè} imply that $\{(u_R,v_R)\}$ is bounded in $H^1(C_1)$. Consequently, it is possible to argue as in the proof of Proposition \ref{compactenss} and obtain the existence of a subsequence of $\{(u_R,v_R)\}$ which converges in $\mathcal{C}^2_{loc}(C_1)$ to a solution $(u^1,v^1)$ of \eqref{eqn: system R} in $C_1$, which inherits by $\{(u_R,v_R)\}$ the properties 2)-4) of Proposition \ref{existence thm in bdd cylinders exp}. In light of Corollary \ref{doubling} and Lemma \ref{lem:N le 2 2}, this procedure can be iterated: indeed
\[
H_R(r) \le \frac{H_R(1)}{e^4} e^{4r} \le C e^{4r} \qquad \forall r>1,
\]
so that applying the previous argument we obtain a subsequence of $\{(u_R,v_R)\}$ which converges in $\mathcal{C}^2_{loc}(C_r)$ to a solution $(u^r,v^r)$ of \eqref{eqn: system R} in $C_r$, and inherits by $\{(u_R,v_R)\}$ the properties 2)-4) of Proposition \ref{existence thm in bdd cylinders exp}. A diagonal selection gives the existence of a solution $(u,v)$ of \eqref{eqn: system R} in the whole $C_\infty$, and this solution enjoys the properties 2)-4) of Proposition \ref{existence thm in bdd cylinders exp}.
\end{proof}

\begin{remark}
The monotonicity formulae proved in subsection \ref{sub:monot Neumann} do not apply on $(u,v)$, because passing to the limit we lose the Neumann condition $\pa_x u_R=0= \pa_x v_R$ on $\Sigma_{-R}$.
\end{remark}

In the next Lemma, we show that $(u,v)$ is a solution with finite energy, so that the achievements proved in subsection \ref{sub:monot finite} applies.

\begin{lemma}\label{lem:convergenza a ener finita}
Let $(u,v)$ be the solution found in Proposition \ref{compactness exp}. It results
\begin{equation}\label{finite 1}
\mathcal{E}^{unb}(r):=\int_{C_{(-\infty,r)}} |\nabla u|^2 + |\nabla v|^2 + u^2 v^2 < +\infty \qquad \forall r \in \R
\end{equation}
and
\[
\lim_{r \to -\infty} H(r) = \lim_{r \to -\infty} \int_{\Sigma_r} u^2 +v^2 = 0.
\]
Recall that $\mathcal{E}^{unb}$ has been defined in subsection \ref{sub:monot finite}.
\end{lemma}
\begin{proof}
Let $\{(u_{R_n},v_{R_n})\}$ be the converging subsequence found in Proposition \ref{compactness exp}, which we will simply denote $\{(u_n,v_n)\}$. Since $\{(u_{n}, v_{n})\}$ converges to $(u,v)$ in $\mathcal{C}^2_{loc}(C_\infty)$, it follows that
\[
\lim_{n \to \infty} \left( |\nabla u_{n}|^2 + |\nabla v_{n}|^2 + u_{n}^2 v_{n}^2 \right) \chi_{C_{(-R_n,r)}} = \left( |\nabla u|^2 + |\nabla v|^2 + u^2 v^2 \right) \chi_{C_{(-\infty,r)}} \qquad \text{a.e. in $C_{(-\infty,r)}$},
\]
for every $r>1$. Therefore, applying Corollary \ref{doubling} on $(u_n,v_n)$, Lemma \ref{lem:H(1) limitato 2} and the Fatou lemma, we deduce
\begin{align*}
\mathcal{E}^{unb}(r)  & \le \liminf_{n \to \infty} \int_{C_{(-\infty,r)}} \left( |\nabla u_{n}|^2 + |\nabla v_{n}|^2 + u_{n}^2 v_{n}^2 \right) \chi_{C_{(-R_n,r)}}  \le \liminf_{n \to \infty} E_{R_n}(r) \\
&= \liminf_{n \to \infty} N_{R_n}(r) H_{R_n}(r) \le \liminf_{n \to \infty} 2 \frac{H_{R_n}(1)}{e^4} e^{4r} \le C e^{4r},
\end{align*}
which proves the \eqref{finite 1}. To complete the proof, we firstly note that necessarily $\mathcal{E}^{unb}(r) \to 0$ as $r \to -\infty$, and hence the same holds for $E^{unb}$ (which has been defined in subsection \ref{sub:monot finite}). Assume by contradiction that for a sequence $r_n \to -\infty$ it results $H(r_n) \ge C >0$. We define
\[
\left(\hat u_n(x,y), \hat v_n(x,y) \right) := \frac{1}{\sqrt{H(r_n)}} \left( u(x+r_n,y), v(x+r_n,y)\right).
\]
A direct computation shows that
\[
\int_{C_{(-\infty,0)}} |\nabla \hat u_n|^2 + |\nabla \hat v_n|^2 \le \int_{C_{(-\infty,0)}} |\nabla \hat u_n|^2 + |\nabla \hat v_n|^2 + 2 H(r_n) \hat u_n^2 \hat v_n^2 = \frac{1}{H(r_n)} E^{unb} (r_n) \to 0
\]
as $n \to \infty$. Consequently, $(\hat u_n, \hat v_n)$ tend to be a pair of constant functions of type $(\hat u,\hat v)$ with $\hat u= \hat v$ (this follows from the symmetries of $(u,v)$). As
\[
C \int_{C_{(-\infty,0)}} \hat u_n^2 \hat v_n^2 \le  H(r_n) \int_{C_{(-\infty,0)}}\hat u_n^2 \hat v_n^2 \to 0,
\]
necessarily $(\hat u_n, \hat v_n) \to (0,0)$ almost everywhere in $C_{(-\infty,0)}$. This is in contradiction with the fact that $\int_{\Sigma_0} \hat u_n^2 + \hat v_n^2 = H(r_n) \ge C$.
\end{proof}

So far we proved that the solution $(u,v)$, found in Proposition \ref{compactness exp}, enjoys properties 1)-5) of Theorem \ref{thm: main thm 2}, and is such that $H(r) \to 0$ as $r \to -\infty$. The previous Lemma enables us to apply the achievements of subsection \ref{sub:monot finite} for $E^{unb}, H, N^{unb}$ and $\mathfrak{N}^{unb}$ (which we consider referred to the solution $(u,v)$ found in Proposition \ref{compactness exp}), and permits to complete the description of the growth of $(u,v)$, points 6)-7) of Theorem \ref{thm: main thm 2}.

\begin{lemma}
Let $(u,v)$ be the solution found in Proposition \ref{compactness exp}. It results
\[
\lim_{r \to +\infty} N^{unb}(r) = 1.
\]
\end{lemma}

\begin{proof}
Let $\{(u_{R_n},v_{R_n})\}$be the converging subsequence found in Proposition \ref{compactness exp}, , which we will simply denote $\{(u_n,v_n)\}$. Firstly, arguing as in the proof of the previous Lemma, we note that by the $\mathcal{C}^2_{loc}(C_\infty)$ convergence of $(u_{n},v_{n})$ to $(u,v)$ it follows that
\[
N^{unb}(r) \le \liminf_{n \to \infty} N_{R_n}(r) \le 2 \qquad \forall r \in \R,
\]
thanks to the Fatou lemma. This, together with the symmetries of $(u,v)$, permits to use Lemma \ref{cor estimate for N finite}, which gives $\lim_{r \to +\infty}N^{unb}(r) \ge 1$. To complete the proof, it is sufficient to show that $\lim_{r \to +\infty} N^{unb}(r) \le 1$. For any $r>0$, let
\[
f_n(r):= \frac{\int_{C_r} u_n^2 v_n^2 }{H_{R_n}(r)}, \qquad g_n(r):= \frac{\int_{\Sigma_r \cup \Sigma_{-r}} u_n^2 v_n^2}{H_{R_n}(r)},
\]
and let $f$ and $g$ the same quantities referred to the solution $(u,v)$. Observe that $f_n,g_n,f$ and $g$ are continuous and nonnegative. The uniform convergence of $(u_n,v_n)$ to $(u,v)$ implies that $f_n \to f$ and $g_n \to g$, as $n \to \infty$, uniformly on compact intervals. By definition,
\begin{equation}
f_n(r) \le \frac{1}{2} N_{R_n}(r) \le 1 \qquad \forall r >0.
\end{equation}
whenever $R_n \ge r$. We claim that $g \in L^1(\R^+)$. Indeed, by the monotonicity of $H$ and Proposition \ref{prp: Almgren k finite}, it follows that
\[
\int_0^r g(s) \, \de s  = \int_{0}^r \frac{\int_{\Sigma_s} u^2 v^2  }{H(s)}\, \de s + \int_{-r}^0 \frac{\int_{\Sigma_s} u^2 v^2  }{H(-s)}\, \de s \le \int_{-r}^r \frac{\int_{\Sigma_s} u^2 v^2  }{H(s)}\, \de s \le \int_{-\infty}^r \frac{\int_{\Sigma_s} u^2 v^2  }{H(s)}\, \de s \le N^{unb}(r),
\]
for every $r>0$. Let $r>0$; it is possible to refine the computation on Lemma \ref{lem: N less 2} to obtain
\[
N_{R_n}(r) \le 1 + f_n(r) + \frac{\int_{C_{(-R_n,-r)}} u_n^2 v_n^2 }{H_{R_n}(r)} \le 1 + f_n(r) + \frac{E_{R_n}(-r)}{H_{R_n}(r)}
\]
Therefore, using again the Fatou lemma we deduce
\[
N^{unb}(r) \le \liminf_{n \to \infty} N_{R_n}(r) \le 1 + f(r) + \liminf_{n \to \infty} \frac{E_{R_n}(-r)}{H_{R_n}(r)},
\]
and to complete the proof we will show that
\begin{equation}\label{eq7}
\lim_{r \to \infty} \left( f(r) + \liminf_{n \to \infty} \frac{E_{R_n}(-r)}{H_{R_n}(r)} \right) = 0.
\end{equation}
Firstly, we note that
\[
\liminf_{n \to \infty} \frac{E_{R_n}(-r)}{H_{R_n}(r)} = \liminf_{n \to \infty} \frac{N_{R_n}(-r) H_{R_n}(-r)}{H_{R_n}(r)} \le 2 \liminf_{n \to \infty} \frac{H_{R_n}(-r)}{H_{R_n}(r)}.
\]
From the $\mathcal{C}^2_{loc}(C_\infty)$ convergence of $(u_n,v_n)$ to $(u,v)$ it follows
\[
2\liminf_{n \to \infty} \frac{H_{R_n}(-r)}{H_{R_n}(r)} = 2 \frac{H(-r)}{H(r)} \to 0 \qquad \text{as $r \to +\infty$}
\]
where we used Lemma \ref{lem:convergenza a ener finita} and the fact that $H(r)>H(0)>0$ for every $r>0$. For the \eqref{eq7} it remains to prove that $f(r) \to 0$ as $r \to +\infty$. Having observed that $\lim_{r \to +\infty} N(r) \ge 1$ and that $g \in L^1(\R^+)$, it is not difficult to adapt the conclusion of the proof of Lemma \ref{lem: N less 1}.
\end{proof}

\section{Systems with many components}\label{sec:k comp}

In this section we are going to prove the existence of entire solutions with exponential growth for the $k$ component system \eqref{eqn: system k comp}. Our construction is based on the elementary limit
\[
	\lim_{d \to +\infty} \Im\left[\left(1+\frac{z}{d}\right)^{d}\right] = e^x \sin y,
\]
which shows that the harmonic function $e^x \sin y$ can be obtained as limit of homogeneous harmonic polynomial. We wish to prove that the same idea applies to solutions of the system \eqref{eqn: system k comp}: there exists an entire solution to \eqref{eqn: system k comp} having exponential growth which can be obtained as limit of entire solutions having algebraic growth.

\subsection{Preliminary results}

We recall some results contained in \cite{BeTeWaWe}. For $d \in \frac{\N}{2}$, let $G_d$ be the rotation of angle $\frac{\pi}{d}$.

\begin{theorem}[Theorem 1.6 of \cite{BeTeWaWe}]\label{thm: BeTeWaWe k-components}
Let $k \ge 2$ be a positive integer, let $d \in \frac{\N}{2}$ be such that
\[
2d=hk \qquad \text{for some $h \in \N$}.
\]
There exists a solution $(u_1^d,\dots,u_k^d)$ to the system \eqref{eqn: system k comp} which enjoys the following symmetries
\begin{equation}\label{eqn: symmetries algebraic}
	\begin{split}
	u_i^d(x,y) &= u_i^d( G_d^k(x,y)) \\
	u_i^d(x,y) &= u_{i+1}^d(G_d(x,y))\\
	u_{k+1-i}^d(x,y) &= u_i^d(x,-y)
	\end{split}
\end{equation}
where we recall that indexes are meant $\mod k$. Moreover
%\begin{equation}\label{eqn:asympt H}
\[	
	\lim_{r \to +\infty} \frac{1}{r^{1+2d}} \int_{\partial B_r} \sum_{i=1}^k \left(u_i^d\right)^2 = b \in (0,+\infty),
\]
%\end{equation}
and
\begin{equation}\label{eqn:asympt N}
	\lim_{r \to +\infty} \frac{r \int_{B_r} \sum_{i=1}^k |\nabla u_i^d|^2 + \sum_{1 \leq i < j \leq k} \left(u_i^d u_j^d\right)^2 }{\int_{\partial B_r} \sum_{i=1}^k \left(u_i^d\right)^2} = d,
\end{equation}
where $B_r$ denotes the ball of center $0$ and radius $r$.
\end{theorem}

The solution $(u_1^d,\dots,u_k^d)$ is modeled on the harmonic polynomial $\Im(z^d)$, as specified by the symmetries \eqref{eqn: symmetries algebraic}. In the quoted statement, the authors modeled their construction on the functions $\Re(z^d)$: it is straightforward to obtain an analogous result replacing the real part with the imaginary one.
\begin{figure}
	 \centering
        \begin{subfigure}[b]{0.45\textwidth}
                \centering
                \includegraphics[width=.9\textwidth]{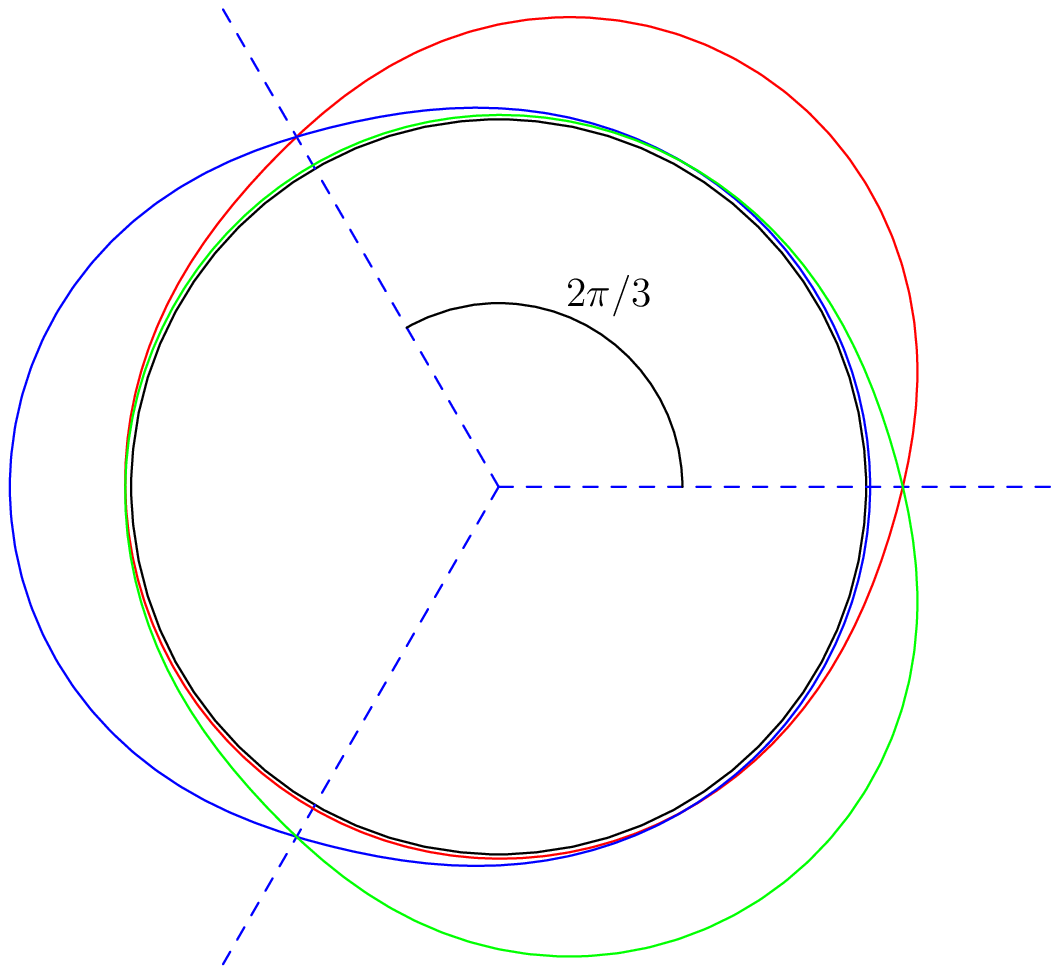}
        \end{subfigure}
        \begin{subfigure}[b]{0.45\textwidth}
                \centering
                \includegraphics[width=.9\textwidth]{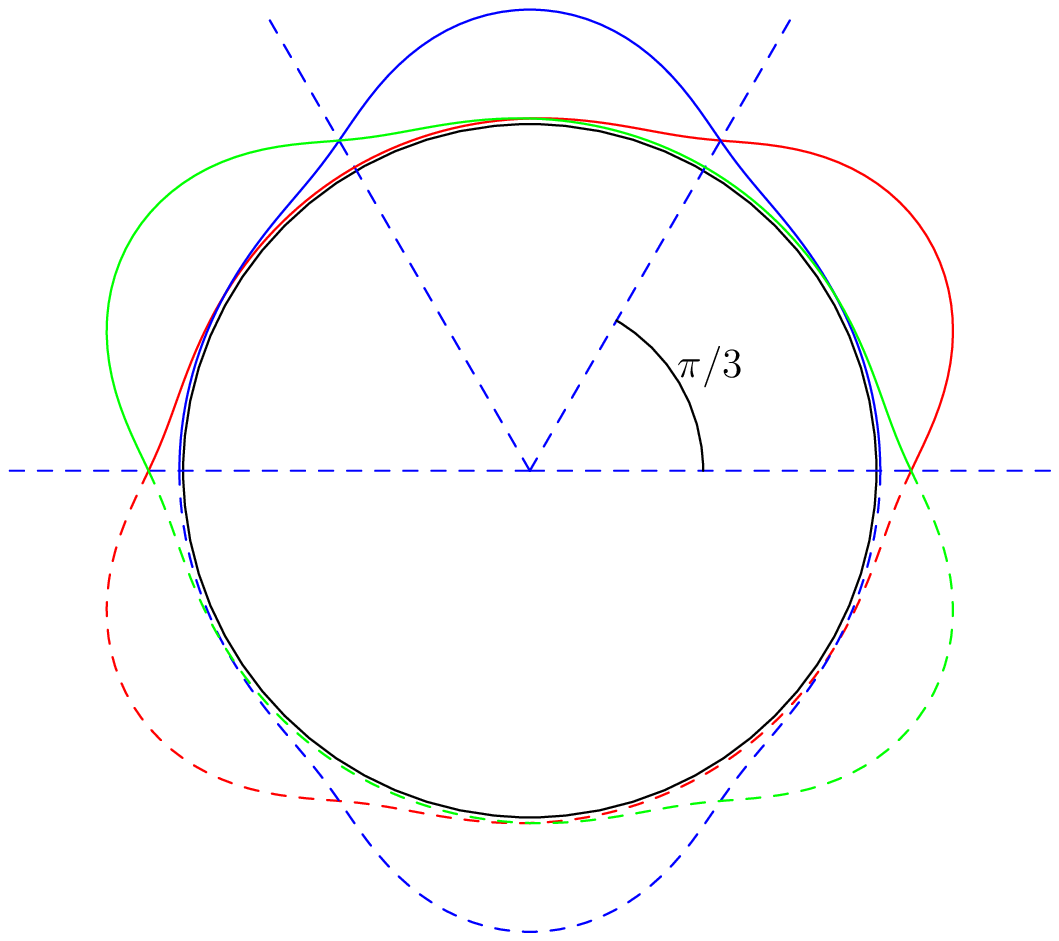}
        \end{subfigure}
        \\
        \begin{subfigure}[b]{0.45\textwidth}
                \centering
                \includegraphics[width=.9\textwidth]{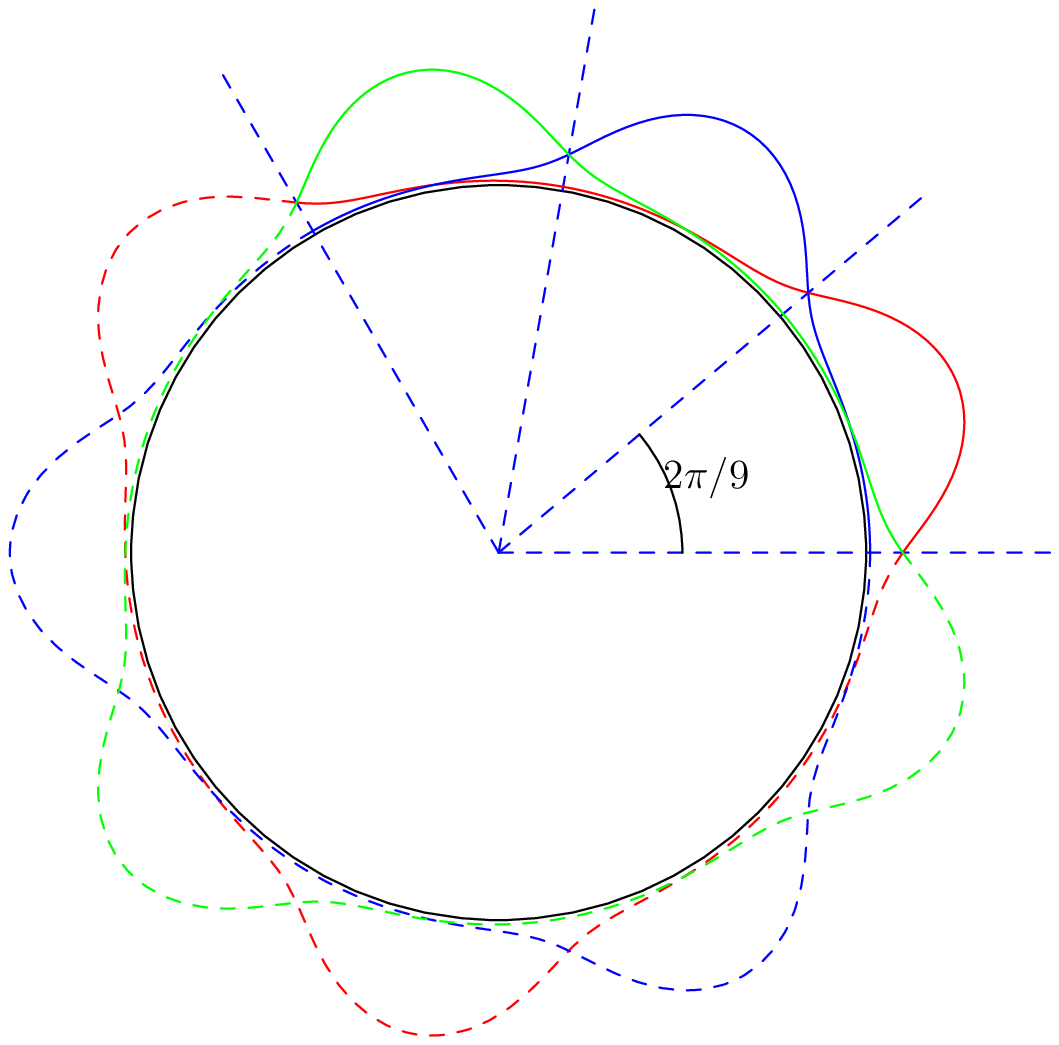}
        \end{subfigure}
        \begin{subfigure}[b]{0.45\textwidth}
                \centering
                \includegraphics[width=.9\textwidth]{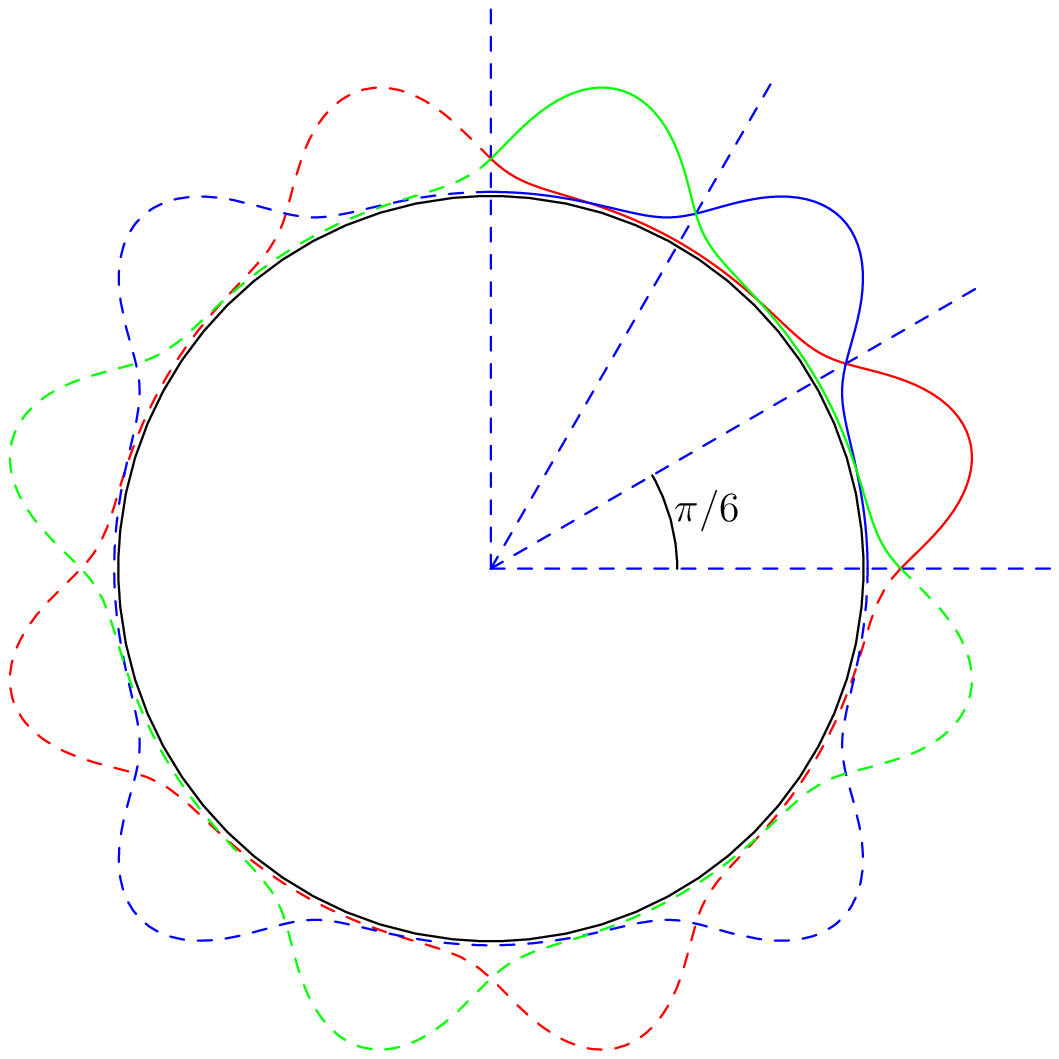}
        \end{subfigure}
  \caption{In the figure we represent some of the solutions obtained in Theorem \ref{thm: BeTeWaWe k-components}. Here the number of components is set as $k =3$: each component is drawn with a different color. On the other hand the periodicity (that is, how many times the patch of $3$-components is replicated in the circle) is given by $h=1$ (up left), $h=2$ (up right), $h=3$ (down left) and $h=4$ (down right), respectively. As a consequence, the growth rate $d$ varies as $d = \frac{3}{2}, 3, \frac{9}{2}, 6$, following the same order.}
\end{figure}
\begin{remark}\label{rem: simmetrie}
We point out that the symmetries \eqref{eqn: symmetries algebraic} implies that $u_1^d$ is symmetric with respect to the reflection with the axis $y= \tan\left(\frac{\pi}{2d}\right) x$.
\end{remark}

For a solution $(u_1,\dots,u_k)$ of system \eqref{eqn: system k comp} in $\R^2$, we introduce the functionals
\begin{equation}\label{eqn: def funct palle}
	\begin{split}
	E^{alg}(r;\Lambda) &:= \int_{B_r} \sum_{i=1}^{k}|\nabla u_i|^2 + \Lambda \sum_{1 \leq i < j \leq k} \left(u_i u_j\right)^2\\
	H^{alg}(r) &:= \frac{1}{r} \int_{\partial B_r} \sum_{i=1}^{k}\left( u_i \right)^2
	\end{split}
\end{equation}
The index $alg$ denotes the fact that these quantities are well suited to describe the growth of $(u_1,\dots,u_k)$ under the assumption that $(u_1,\dots,u_k)$ has algebraic growth. In particular, as proved in Lemma 2.1 of \cite{Fa} and Corollary A.8 of \cite{FaSo} for the case $k=2$, the Almgren quotient
\[
N^{alg}(r;1):= \frac{E^{alg}(r;1)}{H^{alg}(r)}
\]
is bounded in $r \in \R^+$ if and only if $(u_1,\dots,u_k)$ has algebraic growth.

It is not difficult to adapt the proof of Proposition 5.2 in \cite{BeTeWaWe} to obtain the following general  result (in the sense that it holds true for an arbitrary solution of \eqref{eqn: system k comp} in $\R^N$, for any dimension $N \ge 2$).

\begin{proposition}[see Proposition 5.2 of \cite{BeTeWaWe}]\label{prp: Almgren BeTeWaWe}
Let $N \ge 2$,
\[
\Lambda \in \begin{cases}  \left[1, \frac{N}{N-2}\right] & \text{if $N>2$} \\
\left[1,+\infty\right) & \text{if $N=2$},
\end{cases}
\]
and let $(u_1,\ldots,u_k)$ be a solution of \eqref{eqn: system k comp} in $\R^N$; the Almgren quotient
\[
	N^{alg}(r;\Lambda):= \frac{E^{alg} (r;\Lambda)}{H^{alg}(r)} = \ddfrac{r\int_{B_r} \sum_{i=1}^{k}|\nabla u_i|^2 + \Lambda \sum_{1 \leq i < j \leq k} \left(u_i u_j\right)^2}{\int_{\partial B_r} \sum_{i=1}^{k}\left( u_i \right)^2},
\]
is well defined in $(0,+\infty)$ and nondecreasing in $r$.
\end{proposition}
\begin{proof}
We observe that
\begin{equation}\label{eqn: de E lambda}
\begin{split}
\frac{\de }{\de r}E^{alg}(r;\Lambda) &= \frac{\de}{\de r} \left(\frac{1}{r^{N-2}}\int_{B_r} \sum_i|\nabla u_i|^2 + \sum_{i < j} \left(u_i u_j\right)^2\right) + \frac{\de}{\de r} \left(\frac{\Lambda-1}{r^{N-2}}\int_{B_r} \sum_{ i<j} \left(u_i u_j\right)^2\right) \\
& = \frac{2}{r^{N-2}} \int_{\pa B_r} \sum_{i} (\pa_\nu u_i)^2+ \frac{2}{r^{N-1}} \int_{B_r} \sum_{ i < j} \left(u_i u_j\right)^2  \\
& \qquad + \frac{(2-N)(\Lambda-1)}{r^{N-1}} \int_{B_r} \sum_{i<j} u_i^2 u_j^2 + \frac{\Lambda-1}{r^{N-2}}\int_{\pa B_r} \sum_{i<j} u_i^2 u_j^2,
\end{split}
\end{equation}
where we used equation (5.3) in \cite{BeTeWaWe}. Proceeding as in the proof of Proposition 5.2 in \cite{BeTeWaWe}, one gets
\[
\begin{split}
\frac{\de }{\de r}N^{alg}(r;\Lambda) \ge (2+(\Lambda-1)(2-N)) \frac{\int_{B_r} \sum_{i<j} u_i^2 u_j^2  }{r^{N-1} H^{alg}(r)} + \frac{(\Lambda-1)\int_{\pa B_r} \sum_{i<j} u_i^2 u_j^2}{r^{N-2}H^{alg}(r)},
\end{split}
\]
which is $\ge 0$ by our assumption on $\Lambda$.
\end{proof}

\begin{remark}
In \cite{BeTeWaWe} the authors consider the case $\Lambda=1$.
\end{remark}

We work in the plane $\R^2$, so that it is possible to choose $\Lambda=2$ in Proposition \ref{prp: Almgren BeTeWaWe}. We denote $E_d(\cdot;\Lambda)$ and $H_d$ the quantities defined in \eqref{eqn: def funct palle} when referred to the functions $(u_1^d,\ldots,u_k^d)$ defined in Theorem \ref{thm: BeTeWaWe k-components}; also, we denote $\displaystyle N_d(\cdot; \Lambda):= \frac{E_d(\cdot;\Lambda)}{H_d}$. In case $\Lambda=2$, we will simply write $E_d$ and $N_d$ to ease the notation.

\begin{lemma}\label{lem:bound on N}
Let $(u_1^d,\ldots,u_k^d)$ be defined in Theorem \ref{thm: BeTeWaWe k-components}. There holds $\lim_{r \to +\infty} N_d(r)=d$.
\end{lemma}
\begin{proof}
It is an easy consequence of the \eqref{eqn:asympt N} and of Corollary 5.8 in \cite{BeTeWaWe}, where it is proved that for the solution $(u_1^d,\ldots,u_k^d)$ there holds
\[
\lim_{r \to +\infty} \frac{E_d(r;2)}{r^{2d}}= \lim_{r \to +\infty} \frac{E_d(r;1)}{r^{2d}}.
\]
Therefore,
\begin{align*}
\lim_{r \to +\infty} N_d(r) &= \lim_{r \to +\infty} \frac{E_d(r;2)}{H_d(r)} = \lim_{r \to +\infty} \frac{E_d(r;2)}{r^{2d}} \cdot \lim_{r \to +\infty}\frac{r^{2d}}{H_d(r)} \\
&= \lim_{r \to +\infty} \frac{E_d(r;1)}{r^{2d}}\cdot \lim_{r \to +\infty} \frac{r^{2d}}{H_d(r)} = \lim_{r \to +\infty} N_d(r;1) = d. \qedhere
\end{align*}
\end{proof}

As a consequence, the following doubling property holds true:

\begin{proposition}[See Proposition 5.3 of \cite{BeTeWaWe}]\label{prp: doubling BeTeWaWe}
For any $0< r_1 < r_2$ it holds
\[
	\frac{H_d(r_2)}{r_2^{2d}} \leq \frac{H_d(r_1)}{r_1^{2d}}.
\]
\end{proposition}
\begin{proof}
A direct computation shows that
\[
	\frac{\de }{\de r} \log \frac{H_d(r)}{r^{2d}} = \frac{2N_d(r)}{r}-\frac{2d}{r} \leq 0;
\]
an integration gives the thesis.
\end{proof}

Let us consider the scaling
\begin{equation}\label{eqn: first scaling}
    (u_{1,R}^d,\dots, u_{k,R}^d) := \left(\frac{2d}{k H_d(R)}\right)^{\frac12} \left(u_1^d(Rx,Ry), \dots, u_k^d(Rx,Ry)\right),
\end{equation}
where $R$ will be determined later as a function of $d$. We see that
\begin{equation}\label{eqn: system in C_R exp k}
	\begin{cases}
		\displaystyle -\Delta u_{i,R}^d=- \beta_R^d \, u_{i,R}^d \sum_{j\neq i} \left(u_{j,R}^d\right)^2 & \text{in $\R^2$}\\
		\displaystyle \int_{\partial B_1} \sum_{i=1}^k \left(u_{i,R}^d\right)^2 = \frac{2d}{k}
	\end{cases}
\end{equation}
where $\beta_R^d := \frac{k}{2d}H_d(R) R^2$.

\begin{remark}\label{rem: su beta}
As a function of $R$, $\beta^d_R$ is continuous and such that $\beta_R^d \to 0$ if $R \to 0$ and $\beta_R^d \to \infty$ if $R \to \infty$.
\end{remark}

Accordingly with our scaling, we introduce the new Almgren quotient
\[
	N_{d,R}(r):= \frac{E_{d,R} (r)}{H_R(r)} = \ddfrac{r\int_{B_r} \sum_{i=1}^k |\nabla u_{i,R}^d|^2 + 2\beta_R^d \sum_{1 \le  i < j \le k} \left(u_{i,R}^d \, u_{j,R}^d\right)^2}{\int_{\partial B_r} \sum_{i=1}^k \left( u_{i,R}^d \right)^2}.
\]
We point out that $N_{d,R}(r)=N_d(Rr)$, so that from Lemma \ref{lem:bound on N} and the monotonicity of $N_{d}$ we deduce
\begin{equation}\label{eqn:nound on N R}
N_{d,R}(r) \le d \qquad \forall r,R>0,
\end{equation}
for every $d$. By the symmetries, the solution $(u_{1,R}^d,\dots,u_{k,R}^d)$ is $\frac{k\pi}{d}$-periodic with respect to the angular component, thus it is convenient to restrict our attention to the cones
\[
	S^d_r := \left\{(\rho,\theta): \rho \in (0,r), \theta \in \left(0,\frac{k\pi}{d}\right)\right\} \quad \text{and} \quad S^d:= \left\{(\rho,\theta); \rho>0, \theta \in \left(0,\frac{k\pi}{d}\right) \right\}.
\]
The boundary $\partial S^d_r$ can be decomposed as $\partial S^d_r = \partial_p S^d_r \cup \partial _r S^d_r$, where
\[
	\partial_p S^d_r:= (0,r) \times \left\{0,\frac{k\pi}{d}\right\} \quad \text{and} \quad \partial_r S^d_r :=  \{r\} \times \left(0,\frac{k\pi}{d}\right).
\]
Taking into account the periodicity of $(u_{1,R}^d, \ldots, u_{k,R}^d)$, we note that $(u_{1,R}^d, \ldots, u_{k,R}^d)$ has periodic boundary conditions on $\pa_p S^d_r$; furthermore
\begin{equation}\label{eqn: scale 1}
	\begin{split}
		E_{d,R}(r) &= \frac{2d}{k} \int_{S^d_r} \sum_i |\nabla u_{i,R}^d|^2 + 2\beta_R^d \sum_{i < j} \left(u_{i,R}^d \, u_{j,R}^d\right)^2\\
		H_{d,R}(r) &=\frac{2d}{k r}  \int_{\partial_r S^d_r} \sum_i \left( u_{i,R}^d \right)^2 \\
        N_{d,R}(r) & = \ddfrac{r\int_{S^d_r} \sum_i|\nabla u_{i,R}^d|^2 + 2\beta_R^d \sum_{ i < j } \left(u_{i,R}^d \, u_{j,R}^d\right)^2}{\int_{\partial S^d_r} \sum_i \left( u_{i,R}^d \right)^2}.
	\end{split}
\end{equation}

\subsection{A blow-up in a neighborhood of $(1,0)$}

In order to pursue our strategy, we consider the further scaling
\begin{equation}\label{eqn:further}
    (\hat u_{1,R}^d(x,y), \dots, \hat u_{k,R}^d (x,y)) = \frac{\sqrt{\beta_R^d}}{d} \left(u_{1,R}^d\left(1+\frac{x}{d},\frac{y}{d}\right), \dots, u_{k,R}^d\left(1+\frac{x}{d},\frac{y}{d}\right) \right).
\end{equation}
Accordingly, we will consider the scaled domains \(	\hat S^d_r = d \left(S^d_r -(1,0)\right) \) and $\hat S^d = d \left(S^d -(1,0)\right)$ and the respective boundaries. Having in mind to let $d \to \infty$, we observe that this scaling is a blow-up centered in the point $(1,0)$. It is easy to verify that $(\hat u_{1,R}^d,\ldots, \hat u_{k,R}^d)$ solves (see \eqref{eqn: system in C_R exp k})
\begin{equation}\label{eqn: syst R d hat}
	\begin{cases}
		\displaystyle -\Delta \hat u_{i,R}^d=- \hat u_{i,R}^d \sum_{j \ne i} \left(\hat u_{j,R}^d \right)^2 & \text{in $\hat S^d$} \\
		\displaystyle \int_{\partial_r \hat S^d_1} \sum_{i=1}^k \left(\hat u_{i,R}^d\right)^2 = \frac{2\beta_R^d}{k},
	\end{cases}
\end{equation}
with suitable periodic conditions on $\partial \hat S^d$. A direct computation shows that from \eqref{eqn: scale 1} it follows
%\begin{equation}\label{eq1}
\[
 	N_{d,R}(r) =  d \ddfrac{r \int_{\hat S^d_r} \sum_i |\nabla \hat u_{i,R}^d|^2 + 2 \sum_{i<j} \left(\hat u_{i,R}^d \hat u_{j,R}^d\right)^2}{\int_{\partial_r \hat S^d_r} \sum_i \left(\hat u_{i,R}^d\right)^2},
\]
where in the new coordinates
%\begin{equation}\label{eqn: asymptotic}
\begin{equation}\label{eqn: r}
	r = \sqrt{\left(1+\frac{x}{d}\right)^2+\left(\frac{y}{d}\right)^2}.
\end{equation}
We are then led to define a new Almgren quotient for the scaled functions $(\hat u_{1,R}^d,\ldots, \hat u_{k,R}^d)$:
\[
\begin{split}
   \hat E_{d,R}(r)& := \int_{\hat S^d_r} \sum_{i=1}^k |\nabla \hat u_{i,R}^d|^2 + 2 \sum_{1 \le i<j \le k} \left(\hat u_{i,R}^d \hat u_{j,R}^d\right)^2 \\
    \hat H_{d,R}(r) &:= \frac{1}{r}\int_{\partial_r \hat S^d_r} \sum_{i=1}^k \left( \hat u_{i,R}^d \right)^2 \\
	\hat N_{d,R}(r) &:= \frac{\hat E_{d,R}(r)}{\hat H_{d,R}(r)} = \frac{1}{d} N_{d,R}(r).
	\end{split}
\]
From the equation \eqref{eqn:nound on N R}, we deduce
\begin{equation}\label{eqn: estim on hat N}
\hat N_{d,R}(r) \le 1 \qquad \forall r,R>0, \ \forall d \in \frac{\N}{2}.
\end{equation}

In order to understand the behavior of $(\hat u_{1,R}^d,\ldots, \hat u_{k,R}^d)$ when $d \to \infty$, we fix $R=R(d)$ to get a non-degeneracy condition.

\begin{lemma}\label{lem:choice of R}
For every $d \in \frac{\N}{2}$ there exists $R_d>0$ such that
\[
\hat H_{d,R_d}(1)= \int_{\pa_r \hat S_1^d} \sum_i \left( \hat u_{i,R_d}^d \right)^2 =1.
\]
\end{lemma}
\begin{proof}
By \eqref{eqn: syst R d hat} we know that $\hat H_d(1) = \frac{2\beta_R^d}{k}$, so that we have to find $R_d$ such that $\beta_R^d= \frac{k}{2}$. As observed in Remark \ref{rem: su beta}, this choice is possible.
\end{proof}

We denote $(\hat u_1^d, \ldots, \hat u_k^d):= (\hat u_{1, R_d}^d, \ldots, \hat u_{k,R_d}^d)$, $\hat H_d:= \hat H_{d, R_d}$, $\hat E_d:= \hat E_{d,R_d}$, $\hat N_d:= \hat N_{d,R_d}$ and $\beta^d:= \beta^d_{R_d}$. We aim at proving that , up to a subsequence, the family $\left\{(\hat u_{1}^d,\dots,\hat  u_{k}^d): \ d \in \frac{\N}{2} \right\}$ converges, as $d \to +\infty$, to a solution of \eqref{eqn: system k comp}. To this aim, major difficulties arise from the fact that $\hat S^d_r$ and $\hat S^d$ depend on $d$; in the next Lemma we show that this problem can be overcome thanks to a convergence property of these domains.

\begin{lemma}\label{lem: scaling domini}
For any $r>1$, the sets $\hat S^d_r$ converge to $\R \times (0,k\pi)$ as $k \to +\infty$, in the sense that
\[
    \R \times (0,k\pi) = \inter \left(\bigcap_{n \in \frac{\N}{2}} \bigcup_{d > n} \hat S^d_r \right),
\]
where for $A \subset \R^2$ we mean that $\inter (A)$ denotes the inner part $A$. Analogously,
\[
\displaystyle \R \times (0,k\pi) = \inter \left(\bigcap_{n \in \frac{\N}{2}} \bigcup_{d > n} \hat S^d \right) \quad \text{and} \quad  (-\infty,0) \times (0,k\pi) = \inter \left(\bigcap_{n \in \frac{\N}{2}} \bigcup_{d > n} \hat S^d_1 \right),
\]
and for every $\bar x \in \R$
\[
    (-\infty,\bar x) \times (0,k\pi) = \inter \left(\bigcap_{n \in \frac{\N}{2}} \bigcup_{d > n} \hat S^d_{1+\frac{\bar x}{d}} \right).
\]
\end{lemma}
\begin{proof}
We prove only the first claim. Let $r>1$.
\paragraph{\textbf{Step 1)}} $\displaystyle \R \times (0,k\pi) \subset \bigcap_{n \in \frac{\N}{2}} \bigcup_{d > n} \hat S^d_r$. \\
Let $(x,y) \in \R \times (0,k\pi)$. We show that for every $d \in \frac{\N}{2}$ sufficiently large $(x,y) \in \hat S_r^d$, that is, $\left(1+\frac{x}{d},\frac{y}{d}\right) \in S^d_r$, which means
\[
\sqrt{\left(1+\frac{x}{d}\right)^2 + \left(\frac{y}{d}\right)^2} < r \quad \text{and} \quad \arctan\left( \frac{y}{x+d}\right) \in \left(0,\frac{k\pi}{d}\right).
\]
For the first condition it is possible to choose $d$ sufficiently large, as $r>1$. To prove the second condition, we start by considering $d>-x$, so that $ \arctan\left( \frac{y}{x+d}\right)>0$. Now, provided $d$ is sufficiently large
\[
\arctan \left( \frac{y}{x+d}\right) < \frac{k\pi}{d} \quad \Leftrightarrow \quad y < (x+d) \tan\left(\frac{k\pi}{d}\right).
\]
Since $y < k\pi$, there exists $\eps > 0$ such that $y \leq k(1-\eps)\pi$. Let $\bar d$ be sufficiently large so that
\[
x+d > \left(1-\frac{\eps}{2}\right)d \quad  \text{and} \quad \frac{d}{k\pi} \tan \left(\frac{k\pi}{d}\right) > 1-\frac{\eps}{2}
\]
for every $d > \bar d$. Then
\[
(x+d) \tan\left(\frac{k\pi}{d}\right) >\left(1-\frac{\eps}{2}\right)^2 k\pi > (1-\eps)k\pi \ge y
\]
whenever $d>\bar d$.
\paragraph{\textbf{Step 2)}} $\displaystyle \bigcap_{n \in \frac{\N}{2}} \bigcup_{d > n} \hat S^d_r \subset \R \times [0,k\pi]$.\\
We show that $\left( \R \times [0,k\pi]\right)^c \subset \left( \bigcap_{n \in \frac{\N}{2}} \bigcup_{d > n} \hat S^d_r\right)^c$. If $(x,y) \not \in \R \times [0,k\pi]$, then $y > k\pi$ or $y < 0$. We consider only the case $y> k\pi$; in such a situation
\[
y>k\pi = \lim_{d \to \infty} (x+d) \tan \left(\frac{k\pi}{d}\right),
\]
so that $(x,y) \not \in \hat S^d_r$ for every $d$ sufficiently large.
\end{proof}

\begin{remark}\label{rem:conv boundaries}
As a consequence of the previous result, we see that
\[
\pa_r \hat S^d_1 \to \{0\} \times [0,k\pi] \quad \text{and} \quad \pa_r \hat S^d_{1+\frac{\bar x}{d}} \to \{\bar x\} \times [0,k\pi]
\]
for every $\bar x \in \R$.
\end{remark}

\begin{figure}%[b]{0.9\textwidth}
	 \centering
				\includegraphics[width=0.9\textwidth]{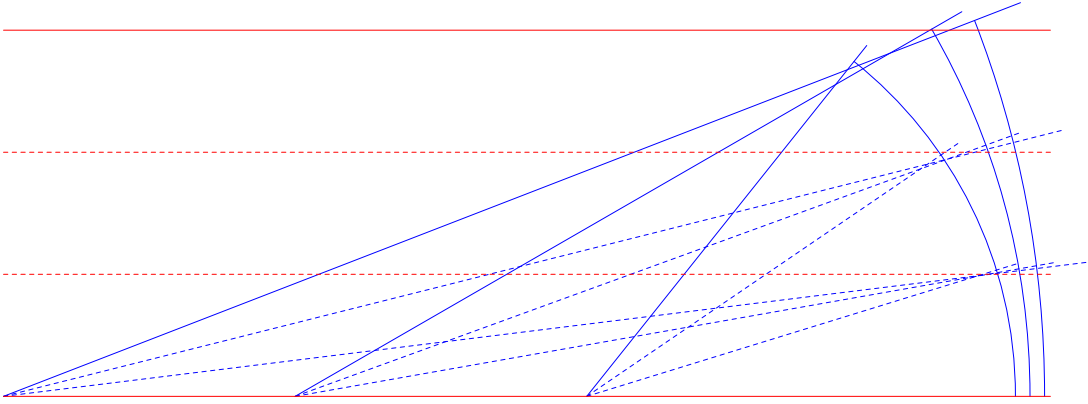}
	\caption{Visualization of the construction in Lemma \ref{lem: scaling domini}. In red the limiting set $\R \times (0,k\pi)$. In blue some of the scaled domains $\hat S^d_r$, for $r>1$.}
\end{figure}

\begin{remark}\label{eqn: asymptotic}
Recall the expression of $r$ in the new variable, given by \eqref{eqn: r}. For every $r>0$ and $d \in \frac{\N}{2}$ there exists $\xi(r,d)$ such that
\[
r=1 + \frac{\xi(r,d)}{d} \quad \Leftrightarrow \quad \xi(r,d)= d(r-1).
\]
Note that for every $(x,y) \in \pa_r \hat S_r^d$ it results $x<\xi(r,d)$. On the contrary, fixing $(x,y) \in \partial_r \hat S_r^d$ there exists $\zeta(d,x,y)$ such that
\[
	r = \sqrt{\left(1+\frac{x}{d}\right)^2+\left(\frac{y}{d}\right)^2} = 1+\frac{x}{d} + \zeta(d,x,y).
\]
In particular, if $y =0$ we have $\zeta(d,x,0) = 0$, while if $y > 0$, $\zeta(d,x,y) \sim d^{-2}$.
\end{remark}

We are ready to prove the convergence of $\{(\hat u_1^d,\dots, \hat u_k^d)\}$ as $d \to \infty$.

\begin{lemma}\label{lem:existence k comp}
Up to a subsequence, $\{(\hat u_1^d,\dots, \hat u_k^d)\}$ converges in $\mathcal{C}^2_{loc}\left( C_\infty \right)$, as $d \to \infty$, to a nontrivial solution $(\hat u_1,\dots, \hat u_k)$ of \eqref{eqn: system k comp}. This solution, which is $k\pi$-periodic in $y$, enjoys the symmetries
\[
\hat u_{i+1}(x,y)= \hat u_i\left(x,y-\pi\right) \quad \text{and} \quad \hat u_1\left(x,y +\frac{\pi}{2}\right) = \hat u_1\left(x,y-\frac{\pi}{2}\right)
\]
\end{lemma}
\begin{proof}
From Proposition \ref{prp: doubling BeTeWaWe} and Lemma \ref{lem:choice of R}, we deduce that for any $r \geq 1$ and $d$ the inequality
\[
    \frac{\hat H_{d}(r) }{r^{2d}} = \frac{\beta^d H_{d}(r) }{d^2 r^{2d}} \leq \frac{\beta^d}{d^2} H_{d}(1) =  \hat H_{d}(1) = 1
\]
holds. For every $x>0$, let $r = 1 + \frac{x}{d}$; for every $d$ sufficiently large, we have
\begin{equation}\label{eqn: estimate on hat H}
 \hat H_{d}\left(1+\frac{x}{d} \right)\leq \left(1+\frac{x}{d}\right)^{2d} \leq 2e^{2 x}
\end{equation}
Recalling the \eqref{eqn: estim on hat N} (which we apply for $R=R_d$), we deduce
\begin{equation}\label{estim on hat E}
\hat E_d\left(1+\frac{x}{d}\right) = \hat N_d\left(1+\frac{x}{d}\right) \hat H_d\left(1+\frac{x}{d}\right) \le 2 e^{2x}
\end{equation}
for every $d$ sufficiently large. Recall that $(\hat u_1^d,\dots,\hat u_k^d)$ can be extended by angular periodicity in the whole plane $\R^2$. Let us introduce
\[
T_r^d:= \left\{(\rho,\theta): \rho<r, \ \theta \in \left( -\frac{\pi}{d}, (k+1)\frac{\pi}{d}\right)\right\} \supset S_r^d,
\]
and let $\hat T_r^d:=d\left( T_r^d-(1,0)\right) \supset \hat S_r^d$. Suitably modifying the argument in Lemma \ref{lem: scaling domini}, it is not difficult to see that
\[
\inter \left(\bigcap_{n \in \frac{\N}{2}} \bigcup_{d > n} \hat T_{1+\frac{\bar x}{d}}^d \right) = (-\infty,\bar x) \times (-\pi,(k+1)\pi)
\]
for every $\bar x \in \R$. Hence, let $B$ an open ball contained in $\R \times (-\pi,(k+1)\pi)$, and let $x_B := \sup \{x: (x,y) \in B\}$, so that $B \subset (-\infty,x_B+1) \times (-\pi,(k+1)\pi)$. Using the same argument in the proof of Lemma \ref{lem: scaling domini}, it is possible to show that
\[
B \subset \hat T_r^d,
\]
for every $d$ sufficiently large, and by the \eqref{estim on hat E} and the periodicity of $(\hat u_1,\dots,\hat u_k)$ we deduce
\[
\int_B \sum_i |\nabla u_i|^2 \le 3 \hat E_d\left(1+\frac{x_B+1}{d}\right) \le 2 e^{2(x_B+1)}
\]
whenever $d$ is sufficiently large. This, together with \eqref{eqn: estimate on hat H}, implies that $\{(\hat u_{1}^d,\ldots, \hat u_{k}^d)\}$ is uniformly bounded in $H^1(B)$, for every $B \subset \R \times (-\pi,(k+1)\pi)$. By the compactness of the trace operator, this bound provides a uniform-in-$d$ bound on the $L^2(\pa K)$ norm for every compact $K \subset \subset \R \times (-\pi,(k+1)\pi)$, which in turns, due to the subharmonicity of $u_{i}^d$, gives a uniform-in-$d$ bound on the $L^\infty(K)$ norm of $\{(\hat u_{1}^d,\ldots, \hat u_{k}^d)\}$, for every compact set $K \subset \subset \R \times (-\pi,(k+1)\pi)$. The standard regularity theory for elliptic equations guarantees that when $d \to \infty$ then $\{(\hat u_{1}^d,\ldots, \hat u_{k}^d)\}$ converges in $\mathcal{C}^2_{loc}(\R \times (-\pi,(k+1)\pi))$, up to a subsequence, to a function $(\hat u_1, \dots, \hat u_k)$ which is a solution to \eqref{eqn: system k comp}. By the convergence and by the normalization required in Lemma \ref{lem:choice of R}, we deduce that (recall also the convergence of the boundaries $\pa \hat S_1^d$, Remark \ref{rem:conv boundaries})
\[
\int_0^{k\pi} \sum_i \hat u_i(0,y)^2\,\de y = 1;
\]
in particular, $(\hat u_1,\dots,\hat u_k)$ is nontrivial. The $k\pi$-periodicity in $y$ follows directly form the convergence of the domains, Lemma \ref{lem: scaling domini}. By the pointwise convergence of $(\hat u_1^d, \dots, \hat u_k^d)$ to $(\hat u_1,\dots, \hat u_k)$ and by the symmetries of each function $(\hat u_1^d, \dots, \hat u_k^d)$ (see equation \eqref{eqn: symmetries algebraic} and Remark \ref{rem: simmetrie}) we deduce also that
\[
\hat u_{i+1}(x,y)= \hat u_i\left(x,y-\pi\right) \quad \text{and} \quad \hat u_1\left(x,y +\frac{\pi}{2}\right) = \hat u_1\left(x,y-\frac{\pi}{2}\right).
\]
\begin{comment}
The last symmetry comes from the following argument: we obtained $(\hat u_1,\dots,\hat u_k)$ as the limit of $(\hat u_1^d,\dots, \hat u_k^d)$; for every $d$ we observed in Remark \ref{rem: simmetrie} that $\hat u_1^d$ is reflectionally symmetric with respect to the reflection with the ax  $y=\tan\left(\frac{\pi}{2d}\right) x$. By our construction it follows that this ax is transformed into the line
\[
y= d \tan\left(\frac{\pi}{2d}\right) \left( \frac{x}{d} +1 \right)
\]
in the change of variable given by \eqref{eqn:further}. As a consequence the function $\hat u_1$ turns out to be symmetric with respect to the reflection with the ax $y= \frac{\pi}{2}$ (we recall that we extend $(\hat u_1,\dots, u_k$ in $\R^2$ by $k\pi$-periodicity in $y$).
\end{comment}
\end{proof}

\subsection{Characterization of the growth of $(\hat u_1,\dots, \hat u_k)$}

So far we proved the existence of a solution $(\hat u_1,\dots,\hat u_k)$ of \eqref{eqn: system k comp} which enjoys the properties 1) and 2) of Theorem \ref{main theorem 2}. In this subsection, we are going to complete the proof of the quoted statement, showing that $(\hat u_1,\dots,\hat u_k)$ enjoys also the properties 3)-5). We denote as $\hat {\mathcal{E}}, \hat E, \hat H$ and $\hat N$ the quantities $\mathcal{E}^{unb}, E^{unb}, H$ and $N^{unb}$ introduced in subsection \ref{sub:monot finite} when referred to the function $(\hat u_1,\ldots,\hat u_k)$. Firstly, we show that $(\hat u_1,\dots,\hat u_k)$ has finite energy, point 3) of Theorem \ref{main theorem 2}, and that $\hat H(x) \to -\infty$ as $x \to -\infty$.

\begin{lemma}\label{lem: estim on E con Fatou}
For every $x \in \R$ there holds $\hat{\mathcal{E}}(x)<+\infty$. In particular
\[
	\hat{\mathcal{E}}(x) \leq \liminf_{d \to \infty} \hat{\mathcal{E}}_d\left(1+\frac{x}{d}\right) \quad \text{and} \quad 	\hat{E}(x) \leq \liminf_{d \to \infty} \hat{E}_d\left(1+\frac{x}{d}\right).
\]
Furthermore, $\displaystyle \lim_{x \to -\infty} \hat H(x) = 0$.
\end{lemma}
\begin{proof}
By the $\mathcal{C}^2_{loc}(\R^2)$ convergence of $(\hat u_1^d,\dots,\hat u_k^d)$ to $(\hat u_1,\dots, \hat u_k)$ and by the convergence properties of the domains $\hat S^d_{1+\frac{x}{d}}$, Lemma \ref{lem: scaling domini}, we deduce
\[
\lim_{d \to \infty} \left(\sum_i |\nabla \hat u_i^d|^2 + \sum_{i<j} \left(\hat u_i^d \hat u_j^d\right)^2\right) \chi_{\hat S_{1+\frac{x}{d}}^d} = \left(\sum_i |\nabla \hat u_i|^2 + \sum_{i<j} \left(\hat u_i \hat u_j\right)^2 \right) \chi_{C_{(-\infty,x)}} \qquad \text{a. e. in $C_\infty$},
\]
for every $x \in \R$. As a consequence, we can apply the Fatou lemma obtaining
\[
\hat{\mathcal{E}}(x) \le \liminf_{d \to \infty} \hat{\mathcal{E}}_d\left(1+\frac{x}{d}\right) \le 2 e^{2x},
\]
where the uniform boundedness of $\hat {\mathcal{E}}_d\left(1+\frac{x}{d}\right)$ comes from \eqref{estim on hat E}. To prove that $\hat H(x) \to 0$ as $x\to -\infty$, we can proceed with the same argument developed in Lemma \ref{lem:convergenza a ener finita}.
\end{proof}

In light of the previous Lemma, the monotonicity formulae proved in subsection \ref{sub:monot finite} applies for $\hat{\mathcal{E}}, \hat E, \hat H$ and $\hat N$.

\begin{lemma}\label{cor: stima su N k}
There holds
\[
\lim_{x \to +\infty} \hat N(x) =1.
\]
\end{lemma}
\begin{proof}
By Proposition \ref{prp: Almgren k finite}, we know that $\hat N$ is nondecreasing in $x$, and thanks to the symmetries of $(\hat u_1,\dots,\hat u_k)$, see Lemma \ref{lem:existence k comp}, Lemma \ref{cor estimate for N finite} implies that $\lim_{x \to +\infty} \hat N(x) \ge 1$. It remains to show that this limit is smaller then $1$. This follows from the estimates of Lemma \ref{lem: estim on E con Fatou} and from the strong convergence of $(\hat u_1^d,\dots, \hat u_k^d) \to (\hat u_1,\dots, \hat u_k)$, which implies that $\hat H_d \left(1+\frac{x}{d}\right) \to \hat H(x)$ as $d \to \infty$: therefore, for every $x \in \R$
\[
\hat N(x) = \frac{\hat E(x)}{\hat H(x)} \le \frac{\liminf_{d \to \infty} \hat E_d(x)}{\lim_{d \to \infty} \hat H_d(x)} = \liminf_{d \to \infty} \hat N_d(x) \le 1,
\]
where we used the \eqref{eqn: estim on hat N}.
\end{proof}

In light of this achievement, we can apply Corollary \ref{doubling finite} to complete the proof of point 5) of Theorem \ref{main theorem 2}. The fact that $\gamma>0$ follows by Lemmas \ref{cor: stima su N k} and \ref{cor estimate for N finite}):
\[
\lim_{r \to +\infty}\frac{\hat H(r)}{e^{2r} }= \lim_{r \to +\infty}\frac{\hat E(r)}{e^{2r}} \cdot \lim_{r \to +\infty}\frac{1}{\hat N(r)} >0.
\]

\begin{remark}
With a similar construction, it is possible to obtain the existence of solutions to \eqref{eqn: system k comp} in $\R^2$ modeled on $\cosh x \sin y$. To do this, we can first construct solutions of \eqref{eqn: system k comp} having algebraic growth defined outside the ball of radius $1$, with homogeneous Neumann  boundary conditions on $\pa B_1$. This can be done suitably modifying the proof of Theorem 1.6 in \cite{BeTeWaWe}. Then, performing a new blow-up in a neighborhood of $(1,0)$, we can obtain a solution of \eqref{eqn: system k comp} defined in $\R_+^2$, with homogeneous Neumann condition on $\{x=0\}$; this solution can be extended by even-symmetry in $x$ in the whole $\R^2$.
\end{remark}

\section{Asymptotics of solutions which are periodic in one variable}\label{sec: blow-down}

In this section we prove Theorem \ref{thm: blow-down 2}.

\begin{proof}[Proof of Theorem \ref{thm: blow-down 2}]
Let us start with case ($i$). First of all, let us recall that, since the solution $(u,v)$ is non trivial, $N(0) > 0$: in particular, from point ($i$) of Corollary \ref{doubling finite} it follows that $H(r) \to + \infty$ as $r \to +\infty$. Let us consider the shifted functions
\[
    (u_R(x,y),v_R(x,y)) := \frac{1}{\sqrt{H(R)}} (u(x+R,y), v(x+R,y))
\]
which solve the system
\[
    \begin{cases}
    - \Delta u_R = - H(R) u_R v_R^2 &\text{in $C_\infty$}\\
    - \Delta v_R = - H(R) u_R^2 v_R &\text{in $C_\infty$}\\
    \displaystyle \int_{\Sigma_{0} } u_R^2+v_R^2 = 1
    \end{cases}
\]
and share the same periodicity of $(u,v)$. We introduce
\[
	\begin{split}
		E_R(r) := \int_{C_{(-\infty,r)}} |\nabla u_R|^2+|\nabla_R|^2 + 2H(R) u_R^2 v_R^2,\\
		H_R(r) := \int_{\Sigma_r } u_R^2+v_R^2 \quad \text{and} \quad N_R(r) := \frac{E_R(r)}{H_R(r)}.
	\end{split}
\]
It is easy to see that
\[
    \begin{split}
    E_R(r) &= \frac{1}{H(R)} E^{unb}(r+R)\\
    H_R(r) &= \frac{1}{H(R)} H(r+R)
    \end{split} \quad \Rightarrow \quad N_R(r) = N^{unb}(r+R)
\]
for any $r$ (recall that $E^{unb}$ and $N^{unb}$ have been defined in subsection \ref{sub:monot finite}). We point out that, by definition, $N_{R_1}(r) \le N_{R_2}(r)$ for every $R_1<R_2$. Furthermore, $N_{R}(r) \le N(+\infty)$ for every $r,R$ and $N_R(r) \to N(+\infty)$ as $R \to \infty$ for every $r \in \R$. Therefore, $N_R$ tends to the constant function $N(+\infty)$ in $L^1_{\loc}(\R)$.

Thanks to the normalization condition $H_R(0) = 1$ and the uniform bound $N_R(r) < N(+\infty)$, applying Corollary \ref{doubling finite} (see also Remark \ref{rem:vale tutto sul beta problema}) we deduce that $H_R(r)$ is uniformly bounded in $R$ for every $r>0$. Consequently, also $E_R(r)$ is uniformly bounded in $R$ for every $r>0$, and this reveals that the sequence $(u_R,v_R)$ is uniformly bounded in $H^1_{\loc}(C_\infty)$ and, by standard elliptic estimates, in $L^{\infty}_{\loc}(C_\infty)$. From Theorem 2.6 of \cite{Wa} (it is a local versione of Theorem 1.1 of \cite{NoTaTeVe}), we evince that the sequence $(u_R,v_R)$ is uniformly bounded also in $\mathcal{C}^{0,\alpha}_{\loc}(C_\infty)$ for any $\alpha \in (0,1)$. Consequently, up to a subsequence, $(u_R,v_R)$ converges in $\mathcal{C}^0_{loc}(C_\infty)$ and in $H^1_{loc}(C_\infty)$ to a pair $(\Psi^+,\Psi^-)$, where $\Psi$ is a nontrivial harmonic function (this is a combination of the main results in \cite{NoTaTeVe} and \cite{DaWaZh}). By the convergence, $\Psi$ has to be $2\pi$-periodic in $y$.

Firstly, we prove that $H(r;\Psi) \to 0$ ar $r \to -\infty$, so that the results of subsection \ref{sub:monot har finite} hold true for $\Psi$. As already observed, $N_R(r)  \ge N_{\bar R} (r) $ for every $r \in \R$, for every $R>\bar R$. By the expression of the logarithmic derivative of $H_R$, see Corollary \ref{doubling finite} (see also Remark \ref{rem:vale tutto sul beta problema}) we have
\[
\frac{\de }{\de r} \log H_R(r) = 2N_R(r) \ge 2 N_{\bar R} (r)= \frac{\de}{\de r} \log H_{\bar R}(r) \qquad \forall r.
\]
As a consequence, taking into account that $H_R(0)=1$ for every $R$, for every $r<0$ it results
\[
\frac{H_R(0)}{H_R(r)} \ge \frac{H_{\bar R}(0)}{H_{\bar R}(r)} \quad \Leftrightarrow \quad H_{\bar R}(r) \ge H_R(r) \qquad \forall R>\bar R.
\]
Passing to the limit as $R \to +\infty$, by the $\mathcal{C}^0_{\loc}(\R^2)$ convergence of $(u_R,v_R)$ to $(\Psi^+,\Psi^-)$ it follows that $H_{\bar R}(r) \ge H(r;\Psi)$, which gives $H(r;\Psi) \to 0$ as $r \to -\infty$ in light of our assumption on $(u,v)$.

Using again the expression of the logarithmic derivative of $H_R$ and $H(\cdot;\Psi)$, we deduce
\[
	\log \frac{H_R(r_2)}{H_R(r_1)}= 2 \int_{r_1}^{r_2} N_R(s) \de s \quad \text{and} \quad \log \frac{H(r_2;\Psi)}{H(r_1;\Psi)}= 2 \int_{r_1}^{r_2} N(s;\Psi) \de s,
\]
where $r_1<r_2$. The left hand side of the first identity converges to the left hand side of the second identity; recalling that $N_R(r) \to N(+\infty)$ for every $r$, we deduce
\[
\int_{r_1}^{r_2} N(s;\Psi) \de s =	\lim_{R \to +\infty} \int_{r_1}^{r_2}  N_R(s) \de s = N(+\infty) (r_2-r_1)  \quad \Rightarrow \quad \frac{1}{r_2-r_1} \int_{r_1}^{r_2} N(s;\Psi)\,\de s = N(+\infty).
\]
for every $r_1<r_2$. It is well known that, being $N(\cdot;\Psi) \in L^1_{\loc}(\R)$, the limit as $r_2 \to r_1$ of the left hand side converges to $N(r_1;\Psi)$ for almost every $r_1 \in \R$. Hence, $N(r;\Psi) = N(+\infty)$ for every $r \in \R$.  We are then in position to apply Proposition \ref{thm: Almgren per armoniche}:
\[
	 \lim_{R \to +\infty} N(R) =\lim_{R \to +\infty} N_R(0) = N(0;\Psi) = d \in \N \setminus \{0\},
\]
and $\Psi(x,y)  = \left[C_1 \cos(d y) + C_2 \sin(d y)\right]e^{d x}$ for some constant $C_1,C_2 \in \R$.

\medskip

As far as case ($ii$) is concerned, for the sake of simplicity we assume $a=0$. One can repeat the proof with minor changes replacing $E^{unb}$ and $N^{unb}$ with $E^{sym}$ and $N^{sym}$ (which have been defined in subsection \ref{sub:monot Neumann}). The unique nontrivial step consists in proving that in this setting $H(r;\Psi) \to 0$ as $r \to -\infty$. To this aim, we note that, as before,
\[
H_R(r) \le H_{\bar R}(r) \qquad \forall R> \bar R,
\]
for every $r> -\bar R$. In particular, if $r>1 -\bar R$, by Proposition \ref{prp:Almgren Neumann} and Corollary \ref{doubling} we deduce
\[
H_R(r) \le H_{\bar R}(r) = \frac{H(r+\bar R)}{H(\bar R)} \le \frac{e^{2N(1)(r+\bar R)}}{e^{2N(1)\bar R}} = e^{2 N(1) r} \qquad \forall R> \bar R,
\]
for every $r>1-\bar R$. Passing to the limit as $R \to +\infty$, by $\mathcal{C}^0_{\loc}(\R^2)$ convergence we obtain
\[
H(r;\Psi) \le e^{2N(1)r} \qquad \forall r \in \R,
\]
which yields $H(r;\Psi) \to 0$ as $r \to -\infty$.
\end{proof}

\appendix
\section{}

We start with the following version of the parabolic minimum principle, which we used in the proof of Proposition \ref{existence thm in bdd cylinders}.

\begin{lemma}\label{lem:parabolic_max_principle}
Let $N \ge 2$, let $\Omega = (a,b) \times \Omega' \subset \R^N$ be open and connected, let $c\in L^\infty(\Omega)$ and let $w \in H^1(\Omega)$ be such that
\[
\begin{cases}
w_t-\Delta w \geq c(x) w & \text{in $[0,T] \times \Omega$}\\
w \ge 0 & \text{on $\{0\} \times \overline{\Omega}$} \\
w \ge 0 & \text{on $(0,T) \times (a,b) \times \pa \Omega'$},
\end{cases}
\]
and $w$ has $(b-a)$-periodic boundary condition on $\{a,b\} \times \Omega'$. Then $w \ge 0$.
\end{lemma}
\begin{proof}
Let $J(t):= \frac{1}{2}\int_{\Omega}  (w^-)^2$. A direct computation shows that $J'(t) \le 2 \|c\|_{L^\infty(\Omega)} J(t)$, where we used the boundary conditions. Consequently,
\[
J(t) \le J(0) e^{2 \|c\|_{L^\infty(\Omega)}t}=0 \qquad \forall t \in [0,T]
\]
where the last identity follows by the initial condition.
\end{proof}

\begin{remark}
Note that we do not require anything about the sign of $c$.
\end{remark}

In sections \ref{sec:proof1} and \ref{sec:sec2}, we exploited many times the following properties of the trace operators.

\begin{theorem}\label{teo:compactness traces}
For $a<b$ real numbers, let $C_{(a,b)}= (a,b) \times \S_k$ be a bounded cylinder. The trace operator $Tr: u \in H^1(C_{(a,b)}) \mapsto u|_{\Sigma_a \cup \Sigma_b} \in L^2(\Sigma_a \cup \Sigma_b)$ is compact.
\end{theorem}

\begin{proof}
Let $(u_n) \subset H^1(C_{(a,b)})$ be such that $u_n \rightharpoonup 0$. We show that $u_n|_{\Sigma_a \cup \Sigma_b} \to 0$ in $L^2(\Sigma_a \cup \Sigma_b)$. For the sake of simplicity we consider the case $a=0$ and $b=1$. Let $w(x,y):= x(x-1)$. We note that $\pa_\nu w= 1$ on $\Sigma_0 \cup \Sigma_1$. Let
\[
F(x,y) = \nabla w(x,y) = (2x-1,0)  \quad \text{and} \quad g(x,y) = \Delta w(x,y) = 2.
\]
By the divergence theorem
\[
2 \int_{C_{(0,1)}} u_n^2 = \int_{C_{(0,1)}} (\div F) u_n^2 = -2 \int_{C_{(0,1)}} 2u_n F\cdot \nabla u_n + \int_{\Sigma_a \cup \Sigma_b} u_n^2,
\]
so that
\[
\int_{\Sigma_a \cup \Sigma_b} u_n^2 \le 2 \|u_n\|_{L^2(C_{(0,1)}}^2 + 2 \|u_n\|_{L^2(C_{(0,1)}} \| \nabla u_n\|_{L^2(C_{(0,1)}} \to 0
\]
as $n \to \infty$, by the compactness of the Sobolev embedding $H^1(C_{(0,1)} \hookrightarrow L^2(C_{(0,1)})$.
\end{proof}

\begin{corollary}\label{cor:local compactenss traces}
For $a<b$ real numbers, let $C_{(a,b)}= (a,b) \times \S_k$ be a bounded cylinder. The local trace operator $T_{\Sigma_b}: u \in H^1(C_{(a,b)}) \mapsto u|_{\Sigma_b} \in L^2(\Sigma_b)$ is compact.
\end{corollary}
\begin{proof}
It is an easy consequence of Theorem \ref{teo:compactness traces} and of the fact that the linear operator $L_f: \varphi \in L^2(\Sigma_a \cup \Sigma_b) \mapsto f \varphi \in L^2(\Sigma_a \cup \Sigma_b)$ is continuous for every $f \in L^\infty(\Sigma_a \cup \Sigma_b)$. As $T_{\Sigma_b} = L_{\chi_{\Sigma_b}} \circ Tr_{C_{(a,b)}}$, where $\chi_{\Sigma_b}$ is the characteristic function of $\Sigma_b$, $T_{\Sigma_b}$ is compact.
\end{proof}

\noindent \textbf{Acknowledgments:} the authors thank Prof. Alberto Farina, Prof. Susanna Terracini and Prof. Gianmaria Verzini for many valuable discussions related to this problem. The first author is partially supported by PRIN 2009 grant "Critical Point Theory and Perturbative Methods for Nonlinear Differential Equations".

%\bibliography{bibliography}
%\bibliographystyle{abbrv}

\end{document}